\documentclass[final]{siamart1116}

\usepackage{lipsum}
\usepackage{amsfonts,amsmath,centernot,stmaryrd}
\usepackage{centernot}
\usepackage{adjustbox,lipsum}
\usepackage{subcaption}
\usepackage{soul}
\usepackage{cite}
\usepackage{xfrac}
\usepackage{hyperref}

\newcommand{\mcM}{\mathcal{M}}
\newcommand{\mcN}{\mathcal{N}}
\newcommand{\mcP}{\mathcal{P}}
\newcommand{\mcR}{\mathcal{R}}
\newcommand{\mcW}{\mathcal{W}}
\newcommand{\mcZ}{\mathcal{Z}}
\newcommand{\tmcP}{\tilde{\mcP}}
\newcommand{\tmcR}{\tilde{\mcR}}

\newcommand{\tcb}[1]{\textcolor{black}{#1}}
\newcommand{\tcbb}[1]{\textcolor{black}{#1}}

\ifpdf  \DeclareGraphicsExtensions{.eps,.pdf,.png,.jpg}
\else
  \DeclareGraphicsExtensions{.eps}
\fi
\usepackage{amsopn}
\DeclareMathOperator{\diag}{diag}
\usepackage{booktabs}

\mathchardef\mhyphen="2D
\newcommand{\TheTitle}{Convergence in Norm of Nonsymmetric\\Algebraic Multigrid
}

\newcommand{\TheShortTitle}{Convergence in Norm of Nonsymmetric Algebraic Multigrid}
\newcommand{\TheName}{T. A. Manteuffel and B. S. Southworth
}
\newcommand{\TheAddress}{  University of Colorado Boulder,
  (\email{ben.s.southworth@gmail.com}).
}
\newcommand{\TheFunding}{This research was conducted with Government support under and awarded by DoD, Air Force Office of Scientific Research, National Defense
Science and Engineering Graduate (NDSEG) Fellowship, 32 CFR 168a. This work was performed under the auspices of the U.S. Department
of Energy by Lawrence Livermore National Laboratory under Contract DE-AC52-07NA27344 and B600360 and the U.S. Department of Energy
under grant numbers (SC) DE-FC02-03ER25574 and (NNSA) DE-NA0002376\@.
}

\author{\TheName\thanks{\TheAddress}}
\author{
Tom Manteuffel \thanks{Department of Applied Mathematics, University of Colorado Boulder
    (\email{tmanteuf@colorado.edu}, \email{ben.s.southworth@gmail.com}).}
   \and
Ben S. Southworth \footnotemark[2] 
}
\title{{\TheTitle}\thanks{\TheFunding}}
\headers{\TheShortTitle}{\TheName}
\ifpdf\hypersetup{  pdftitle={\TheTitle},
  pdfauthor={\TheName}
}
\fi

\begin{document}
\allowdisplaybreaks

\maketitle
\vspace{1cm}

\begin{abstract}
Algebraic multigrid (AMG) is one of the fastest numerical methods for solving large sparse linear systems. For SPD matrices, convergence
of AMG is well motivated in the $A$-norm,  and AMG has proven to be an effective solver for many
applications. Recently, several AMG algorithms have been developed that are effective on nonsymmetric linear systems.
Although motivation was provided in each case, the convergence of AMG for nonsymmetric linear systems
is still not well understood, and algorithms are based largely on heuristics or incomplete theory.

For multigrid restriction and interpolation operators, $R$ and $P$, respectively, let $I - \Pi:= I - P(RAP)^{-1}RA$ denote the projection
corresponding to coarse-grid correction in AMG. It is invariably the case in
the nonsymmetric setting that $\|\Pi\| = \|I - \Pi\| > 1$ in any known norm. This causes an interesting dichotomy: coarse-grid correction is
fundamental to AMG achieving fast convergence, but, in this case, can actually increase the error. Here, we present a detailed
analysis of nonsymmetric AMG, discussing why SPD theory breaks down in the nonsymmetric setting, and developing a general
framework for convergence of NS-AMG. Classical multigrid weak and strong approximation properties are generalized
to a \textit{fractional approximation property}. Conditions are then developed on $R$ and $P$ to ensure that $\|\Pi\|_{\sqrt{A^*A}}$
is nicely bounded, independent of problem size. This is followed by the development of conditions for two-grid and multilevel
W-cycle convergence in the $\sqrt{A^*A}$-norm.
\end{abstract}


\section{Introduction} \label{sec:intro}

Large, sparse, nonsymmetric linear systems arise in a number of applications involving directed graph Laplacians, Markov chains, and the
discretization of partial differential equations (PDEs). Algebraic multigrid (AMG) is a multilevel iterative method for solving large sparse
linear systems based on projecting the problem into progressively smaller subspaces. AMG is traditionally motivated for symmetric
positive definite (SPD) linear systems and M-matrices \cite{Brandt:1985um,Ruge:1987vg},
and has shown to be a robust and scalable solver for many such problems. Consistent with other approximate direct solvers,
iterative methods, and Krylov methods, convergence theory in the case of SPD matrices is relatively well-understood \cite{Brandt:1985um,
Falgout:2004cs,Falgout:2005hm,MacLachlan:2014di,Notay:2014uc,Ruge:1987vg,Van:2001bw,Vassilevski:2008wd,Vassilevski:2010vy,Zikatanov:2008jp}.
Although AMG solvers have been developed that can be effective on nonsymmetric problems in various settings (for example,
\cite{Sala:2008cv,Manteuffel:2017,fox2018algebraic,air1,air2,Wiesner:2014cy,Yavneh:2012fb,Notay:2000vy,Seibold,Lottes:2017jz}),
few results have been proven regarding convergence of nonsymmetric AMG (NS-AMG). 

Typically in AMG, simple relaxation schemes are used and the focus of theory and algorithm development is on effective and
complementary coarse-grid correction. For a nonsingular matrix $A\in\mathbb{C}^{n\times n}$, a coarse-grid problem is defined by
projecting $A$ into a subspace using {restriction} and {interpolation} operators, $R, P \in\mathbb{C}^{n \times n_c}$, respectively,
and inverting the coarse-grid operator $A_c := R^*AP \in\mathbb{C}^{n_c\times n_c}$. If $A_c$ is too large to invert directly,
AMG is called recursively on the coarse-grid problem. For SPD matrices, convergence is
considered in the so-called \textit{energy-norm} or $A$-norm, $\|\mathbf{x}\|_A^2 = \langle A\mathbf{x},\mathbf{x}\rangle$. 
Letting $R = P$, coarse-grid correction is an orthogonal projection onto the range of $P$ in the $A$-norm. The focus of AMG for
SPD problems is then on building a ``good'' $P$. In the non-SPD setting, $\langle A\mathbf{x},\mathbf{x}\rangle$ is not well
defined. A key implication of this is that coarse-grid correction in NS-AMG is generally a non-orthogonal projection in any known
inner product, which means that it can increase error. This poses an interesting dichotomy: coarse-grid correction is a principle
mechanism by which AMG reduces error, but, in this case, it may also increase error at times. This makes convergence theory
difficult to develop, as any potential increase in error due to coarse-grid correction must be overcome by other means.
\tcbb{Because of this, there is a need for nonsymmetric convergence theory that can motivate how to build $R$ and $P$ in a
compatible sense for a well-posed, nicely bounded (in norm) coarse-grid correction.}

The simplest measure of NS-AMG convergence is the spectral radius of error propagation, which bounds asymptotic
convergence \cite{Notay:2010em,air2,Wiesner:2014cy}. Although the spectral radius can provide
motivation in developing NS-AMG, it is not necessarily indicative of practical performance. Recently, it was suggested that the
field of values is a more appropriate measure \cite{notay2018}, consistent with previous work on nonsymmetric linear
systems as early as \cite{manteuffel1977tchebychev}. A proof of two-grid convergence was given in \cite{Lottes:2017jz}
for nonsymmetric matrices with positive real parts in the \textit{form absolute value} norm.
A significant theoretical framework was used to develop the form absolute value as a generalization of the $A$-norm
for nonsymmetric matrices. However, the norm is difficult to compute or interpret in practice and leaves open questions on
the respective roles of interpolation and restriction in NS-AMG. In \cite{Brezina:2010dm}, the $A$-norm was generalized to the
nonsymmetric setting by considering the $\sqrt{A^*A}$-norm, and sufficient conditions were derived for two-grid convergence.
However, the conditions in \cite{Brezina:2010dm} include an assumption that the non-orthogonal coarse-grid correction is
bounded in norm by some small constant. This assumption is one of the fundamental difficulties with NS-AMG and, again,
leaves open questions on how to build $R$ and $P$ in the nonsymmetric setting.

This paper builds on the nonsymmetric framework developed in \cite{Brezina:2010dm}. \tcbb{Background on the nonsymmetric
setting and a new generalization of multigrid approximation properties is presented in Section \ref{sec:2grid:background},
followed by the development of general conditions on $R$ and $P$ for bounded coarse-grid corrections and
two-grid convergence in Sections \ref{sec:2grid:basis}
and \ref{sec:2grid:stability}. Section \ref{sec:multigrid} extends these results to the multilevel setting, establishing sufficient conditions
for $W$-cycle convergence. Although one of the conditions on $R$ and $P$ is not easy to establish, it offers insight into
the development of AMG methods for nonsymmetric systems. Moreover, this is the first general result 
on convergence in norm of NS-AMG.}\footnote{A reduction-based NS-AMG
method was developed simultaneously with this work in \cite{air1}. There, sufficient conditions are developed for two-grid
convergence of error in the $\ell^2$- and $A^*A$-norms. Results here take a more traditional AMG approach 
(as opposed to reduction based), and develop a more detailed analysis of the multilevel setting.} 
\tcbb{In Section \ref{sec:numerics} several choices of transfer operators and the resulting non-orthogonal coarse-grid corrections
are analyzed numerically for two discretizations of a hyperbolic PDE.}
A discussion on results and their relation to recently developed, effective NS-AMG solvers is given in Section \ref{sec:conclusion}.

\section{Two-grid convergence} \label{sec:2grid}

\subsection{Background, Definitions, and Assumptions}\label{sec:2grid:background}

Multigrid originated in the geometric setting, applied to elliptic differential operators. There, the $A$-norm corresponds with 
the $\mathcal{H}^1$-Sobolev norm, which enforces accuracy of solution values \textit{and} derivatives. This avoids approximate
solutions with large oscillations and non-physical behavior that can occur when minimizing, for example, the $l^2$-norm. Such
behavior is desirable when considering nonsymmetric problems as well, motivating a $\sqrt{A^*A}$- or $\sqrt{AA^*}$-
generalization of the $A$-norm \cite{Brezina:2010dm}. Let $A\in\mathbb{C}^{n\times n}$ be nonsingular with singular value
decomposition (SVD) $A = U\Sigma V^*$ and singular values ordered such that $0< \sigma_1 \leq \sigma_2 \leq \dots \leq \sigma_n$.
Defining $Q := VU^*$, then $\sqrt{A^*A} = QA = V\Sigma V^*$ and $\sqrt{AA^*} = AQ = U\Sigma U^*$.
Because $\sqrt{A^*A}$ and $\sqrt{AA^*}$ are SPD, we can solve $A\mathbf{x}=\mathbf{b}$ by applying classical AMG techniques
to the equivalent (SPD) linear systems
\begin{align}
QA \mathbf{x} &= Q\mathbf{b}, \label{eq:normal1}\\
AQ\mathbf{y} &= \mathbf{b} \mbox{ for } \mathbf{x} = Q\mathbf{y}. \nonumber
\end{align}
Although $Q$ is difficult to form in practice, these systems provide a framework for convergence of NS-AMG.\footnote{
Note that \eqref{eq:normal1} resembles a normal-equation formulation of the problem. However, AMG is typically applied to large,
sparse, ill-conditioned matrices, and solving the normal equations squares the condition number. Because $Q$ is unitary,
the condition number of $QA$ equals that of $A$.} In particular, classical AMG approximation properties can be considered
with respect to SPD matrices $QA$ and $AQ$, corresponding to the right and left singular vectors.  

\tcb{Coarse-grid correction in multigrid approximates the action of $A^{-1}$ with the operator $PA_c^{-1}R^*$; that is, it restricts
the problem to a subspace, inverts the coarse-grid operator $A_c$ in the subspace, and interpolates the result back to the fine
grid. Error propagation of coarse-grid correction is given as a projection onto the range of $P$: 
\begin{align}
I-\Pi :=I - PA_c^{-1}R^*A,\label{eq:cgc}
\end{align}
Here, $I-\Pi$ corresponds to a two-level method, where the coarse-grid operator $A_c$ is inverted exactly. }
Given an interpolation operator $P$, defining $R:=Q^*P$ makes $I - \Pi$ a $QA$-orthogonal coarse-grid correction. In this case, classical AMG theory
applies, and the optimal $P$ with respect to two-grid convergence is given by letting columns of $P$ be the first $n_c$ right singular vectors,
where $n_c$ is the size of the coarse grid \cite{Falgout:2005hm}. It follows that the optimal $R$ then consists of the first $n_c$ left
singular vectors. Thus, in the nonsymmetric development that follows, we consider $P$ that satisfies some approximation property with
respect to $QA$ and $R$ that satisfies some approximation property with respect to $AQ$. Approximation properties on $P$ with respect to
$QA$ ensure that right singular vectors with small singular values are well represented in the range of $P$, denoted 
$\mathbf{R}(P)$, and likewise for $R$, $AQ$, and left singular vectors. {\color{black}The following definition introduces a new
generalization of classical multigrid approximation properties, called a \textit{fractional approximation property} (FAP).

\begin{definition}[Fractional Approximation Property: FAP($\beta,\eta$)]
A transfer operator $P$ is said to have a FAP with respect to the SPD matrix $\mathcal{A}$, with powers $\beta,\eta \geq 0$ and constant 
$K_{P,\beta,\eta}$, if, for every fine-grid vector, 
$\mathbf{v}$, there exists a coarse-grid vector, $\mathbf{v}_c$, such that
\begin{equation*}
\| \mathbf{v} - P\mathbf{v}_c \|^2_{\mathcal{A}^\eta} \leq \frac{K_{P,\beta,\eta}}{\| \mathcal{A} \|^{2\beta-\eta} }\langle \mathcal{A}^{2\beta}  \mathbf{v}, \mathbf{v} \rangle.
\end{equation*}
\end{definition}
The classical multigrid weak approximation property (WAP) is a FAP$(1/2,0)$, the strong approximation property (SAP)  is a
FAP$(1,1)$, and a super strong approximation property (SSAP) is a FAP$(1,0)$. The next result implies relationships
between various approximation properties.

\begin{theorem}\label{th:FAP}  
Let $P$ satisfy a FAP($\beta,\eta$) with respect to $\mathcal{A}$. Then,
\begin{enumerate}
\item $P$ satisfies a FAP($\alpha,\kappa$) for any $0\leq \alpha \leq \beta$ and $\kappa\geq\eta$, 
with constant $K_{P,\alpha,\kappa} \leq K_{P,\beta,\eta}$,
\item If, in addition, $\eta \leq \beta$, then $P$ satisfies a FAP($\beta,\kappa$) for any $0\leq \kappa  \leq \eta \leq \beta$, 
with constant $K_{P,\beta,\kappa} \leq K_{P,\beta,\eta}^2$.
\end{enumerate}
\end{theorem}

\begin{proof}
The first part is found by noting that, for any $\kappa \geq \eta$ and $0\leq \alpha \leq \beta$,
\begin{equation}
\|\mathbf{v} - P\mathbf{v}_c\|_{\mathcal{A}^{\kappa}}^2 \leq \| \mathcal{A}^{\kappa-\eta}\| \|\mathbf{v} - P\mathbf{v}_c\|_{\mathcal{A}^\eta}^2
~~\mbox{and}~~
\langle \mathcal{A}^{2\beta} \mathbf{v}, \mathbf{v}\rangle \| \leq \| \mathcal{A}^{2(\beta-\alpha)}\| \langle \mathcal{A}^{2\alpha}\mathbf{v}, \mathbf{v}\rangle.
\end{equation}
The proof of the second part is found in the Appendix.
\end{proof}

The following relations between well-known multigrid approximation properties follow immediately from Theorem \ref{th:FAP}.

\begin{corollary}[Equivalence of approximation properties]\text{ \\ }
\label{lem:approxprop} Let $\mathcal{A}$ be SPD.
\begin{enumerate}
\item If $P$ satisfies the SSAP (FAP$(1,0)$) with respect to $\mathcal{A}$ with constant $K_S$, 
then $P$  satisfies the WAP (FAP$(1/2,0)$) with respect to $\mathcal{A}$ with constant $K_W \leq K_S$.
\item If $P$ satisfies the SSAP (FAP$(1,0)$) with respect to $\mathcal{A}$ with constant $K_S$, 
then $P$ satisfies the SAP (FAP$(1,1)$) with respect to $\mathcal{A}$ with constant $K_P \leq K_S$.
\item If $P$ satisfies the SAP (FAP$(1,1)$) with respect to $\mathcal{A}$ with constant $K_P$, 
then $P$ satisfies the SSAP (FAP$(1,0)$) with respect to $\mathcal{A}$ with constant $K_S \leq K_P^2$.
\item If $P$ satisfies the SAP (FAP$(1,1)$) with respect to $\mathcal{A}$ with constant $K_P$, 
then $P$ satisfies the WAP (FAP$(1/2,0)$) with respect to $\mathcal{A}$ with constant $K_W \leq K_P^2$.
\end{enumerate}
\end{corollary}

}

{\color{black}
In the discrete setting, for any SPD matrix, $\mathcal{A}$, any full rank transfer operator, $P$, will satisfy
a FAP$(\beta,\eta)$ for some constant $K_{P,\beta,\eta}$. This is only useful if $K_{P,\beta,\eta}$
is relatively small. Moreover, the approximation property must hold with constant independent of 
the problem size, $n$. One can think of $\mathcal{A}$ as a discrete form of a PDE and $P$ as a strategy
for approximating the eigenvectors associated with the smallest eigenvalues values
of $\mathcal{A}$. The goal is for the FAP to hold with a constant that is independent of the 
discretization accuracy of $\mathcal{A}$, which is usually correlated with the problem size, $n$.

In this paper, approximation properties for $P$ will be with respect to $QA$ and approximation properties for $R$ 
will be with respect to $AQ$.
In the multi-level setting, a sequence of transfer operators, say $P_\ell, R_\ell$, are formed and yield a sequence 
of coarse grid operators $A_{\ell+1} = R_{\ell}^*A_\ell P_\ell$.  In the development below, $P_\ell$ is assumed
to have approximation properties with respect to $Q_\ell A_\ell$ and $R_{\ell}$ with respect to $A_\ell Q_\ell$, both with 
constants independent of the grid level, $\ell$, and problem size, $n$. Independent of grid level is somewhat different than 
independent of problem size because the coarse-grid operators no longer need be closely related to the
original PDE. 
}

For SPD systems, satisfying the WAP (FAP$(1/2,0)$) is a necessary and sufficient condition for two-grid convergence \cite{Falgout:2004cs}, and
satisfying the SAP(FAP$(1,1)$) on all levels are sufficient conditions for multilevel convergence \cite{Ruge:1987vg,Vassilevski:2008wd}. 
Nonsymmetric matrices lead to a non-orthogonal coarse-grid correction, which requires stronger conditions for convergence. In particular, it is important
that coarse-grid correction be \textit{stable}, that is, coarse-grid correction can only increase error by some small constant $C_\Pi \geq 1$,
independent of the problem size: 
\begin{definition}[Stability of $\Pi$ in $\mathcal{A}$-norm]
\label{assp:stability}
\begin{align}
 \|\Pi\|_{\mathcal{A}}^2 \leq C_\Pi,
\end{align}
where $C_\Pi \geq 1$ is an $O(1)$ constant, independent of the problem size.
\end{definition}

A natural idea for NS-AMG is to introduce approximation properties on both $R$ and $P$. However, a simple example shows
that building $R$ and $P$ to both satisfy a SAP does not imply stability:
\begin{example}\label{ex:counter}
Let $n_c$ be the size of the coarse-grid problem and $\ell < n_c$ some number such that $\sigma_{\ell} \sim O(1)$. For right
singular vectors $\{\mathbf{v}_i\}$ and left singular vectors $\{\mathbf{u}_i\}$, define 
\begin{align*}
P &:= \left[ \mathbf{v}_1,...,\mathbf{v}_{\ell-1},\mathbf{v}_{\ell+1}, ..., \mathbf{v}_{n_c+1}\right] ,\hspace{4ex}
R := \left[ \mathbf{u}_1,...,\mathbf{u}_{n_c}\right].
\end{align*}
Although $\mathbf{v}_{\ell} \not\in\mathbf{R}(P)$, because $\sigma_{\ell} \sim O(\|A\|)$, $P$ trivially satisfies the SAP for
$\mathbf{v}_{\ell}$ by interpolating the zero vector. Then, it is clear that $P$ satisfies a SAP with respect to $QA$ and
$R$ satisfies a SAP with respect to $AQ$, independent of problem size. However, for the $n_c$th canonical basis vector,
$\mathbf{e}_{n_c}$, $R^*AP\mathbf{e}_{n_c}  = \mathbf{0}$. That is, $R^*AP$ is singular, which implies $\|\Pi\|$ is not well-defined.
\end{example}

Thus, more than two approximation properties are needed for convergence of NS-AMG. In \cite{Brezina:2010dm}, Theorem
\ref{th:2grid} is proven, showing that stability of $\|\Pi\|_{QA}$ and the SAP on $P$ with respect to the $QA$-norm, along with additional relaxation
to account for potential increases in error from coarse-grid correction, are sufficient conditions for two-grid convergence in the
$QA$-norm. In \cite{Brezina:2010dm}, the number of relaxation iterations required to prove convergence scales like the square of
the SAP constant. \tcbb{Here, we show that the number of relaxation iterations can depend on the strength of the approximation property
of $P$. For completeness, the result from \cite{Brezina:2010dm} is repeated.
}

\begin{theorem}[Two-grid $QA$-Convergence (Theorem 2.3, \cite{Brezina:2010dm})] \label{th:2grid}
Let $G $ be the error-propagation operator for $\nu$ iterations of Richardson-relaxation on the normal equations ($A^*A$),
$G := \big ( I - \frac{A^*A}{\|A\|^2}\big )^\nu$, and $(I - \Pi)$
the (non-orthogonal) coarse-grid correction defined by restriction and interpolation operators, $R$ and $P$, respectively (see \eqref{eq:cgc}).
If $P$ satisfies a SAP with respect to the $QA$-norm with constant $K_P$ and coarse-grid correction is stable with
constant $C_\Pi$, then
\begin{align*}
\| (I-\Pi)G\mathbf{e}\|_{QA} \leq \frac{16C_\Pi K_P}{25\sqrt{4\nu +1}} \|\mathbf{e}\|_{QA}.
\end{align*}
Two-grid convergence of NS-AMG in the $QA$-norm follows by performing sufficient iterations of relaxation,
$\nu$, such that $\| (I-\Pi)G\mathbf{e}\|_{QA}<\|\mathbf{e}\|_{QA}$. 
\end{theorem}

{\color{black}

Theorem \ref{th:2grid} assumes that $P$ satisfies a SAP and requires a number of relaxations that grows with the square of 
of the constants $C_{\Pi}K_P$. The next corollary examines the number of relaxations that are sufficient for convergence
if a FAP with a different power is assumed.

\begin{corollary}\label{cor:2grid}
Assume the hypothesis of Theorem \ref{th:2grid}, with the exception that $P$ satisfies a FAP($\beta,1$), $\beta > 1/2$,
with respect to the $QA$-norm with 
constant $K_{P,\beta,1}$. Then,
\begin{align*}
\| (I-\Pi)G\mathbf{e}\|_{QA} \leq \left(\frac{4}{4+(2\beta-1)}\right)^{2} \left(\frac{(2\beta-1)}{4\nu+(2\beta-1)}\right)^{(2\beta-1)/2} C_{\Pi}K_{P,\beta,1} \|\mathbf{e}\|_{QA}.
\end{align*}
Two-grid convergence of NS-AMG in the $QA$-norm follows by performing sufficient iterations of relaxation,
$\nu$, such that $\| (I-\Pi)G\mathbf{e}\|_{QA}<\|\mathbf{e}\|_{QA}$. 
\end{corollary} 
 
 \begin{proof}
 The proof follows from the proof of Theorem \ref{th:2grid} in \cite{Brezina:2010dm} with modifications for the power $\beta$.
 \end{proof}

     Theorem \ref{th:2grid} and Corollary \ref{cor:2grid} are sufficient conditions and are likely not sharp. However, they do expose 
   the importance of the power of the approximation property. If $P$ satisfies a SAP (FAP$(1,1)$), 
   then $\beta = 1$ and then the number of relaxations sufficient to guarantee convergence 
   grows like the square of the constants 
   $C_\Pi K_P$. If $P$ is more accurate and satisfies a FAP$(3/2,1)$, that is, with $\beta = 3/2$,  the number of relaxations
   grows linearly with $C_\Pi K_P$. If $P$ only satisfies a FAP slightly better than a WAP, that is, $\beta > 1/2$, then the number of
   relaxations grows like $(C_\Pi K_P)^{\frac{2}{(2\beta-1)}}$ and can be very large.

  }

Defining a stable coarse-grid correction, with $C_\Pi\sim \mathcal{O}(1)$, is a crux of NS-AMG.  Approximation properties
alone are not sufficient for stability, and stability by definition does not give useful information for building $R$ and $P$, motivating further study
on conditions for stability and two-grid convergence. In particular, we seek conditions on $R$ and
$P$ that give insight to their respective roles in NS-AMG convergence.

The paper proceeds as follows. A basis under which to consider convergence is developed
in Section \ref{sec:2grid:basis}, followed by a proof of sufficient conditions for stability and two-grid convergence in Section \ref{sec:2grid:stability}
(Theorem \ref{th:stability}). Section \ref{sec:multigrid} examines the multilevel case, \tcbb{establishing sufficient conditions for the equivalence 
between two inner products in Section \ref{sec:multigrid:equiv}, and sufficient conditions for $W$-cycle convergence in Section \ref{sec:multigrid:conv}. }
  
  \tcbb{
  In the remainder of this paper, $\beta,\eta$-subscripts in approximation property constants, $K_{P,\beta,\eta}$, are omitted when the 
  meaning is clear. Proofs will make regular use of the following results on equivalent operators, and bounding the action of a
$2\times 2$ block matrix above and below, for which a proof can be found in the Appendix.}

\begin{definition}[Equivalent operators]
Two SPD operators, $A$ and $B$, are said to be \textit{spectrally equivalent} and two general operators, $A$ and $B$, 
\textit{norm equivalent} if there exist constants, $\alpha_s,\beta_s$ and $\alpha_n,\beta_n$, respectively, such that
\begin{align}\label{eqn:equiv}
\alpha_s \leq \frac{\langle A\mathbf{x},\mathbf{x}\rangle}{\langle B\mathbf{x},\mathbf{x}\rangle} \leq \beta_s, \hspace{4ex}
	\alpha_n \leq \frac{\langle A\mathbf{x},A\mathbf{x}\rangle}{\langle B\mathbf{x},B\mathbf{x}\rangle} \leq \beta_n,
\end{align}
denoted $A\sim_s B$ and $A\sim_n B$. For self-adjoint, compact operators on a separable Hilbert space,
$A\sim_n B\implies A\sim_s B$, with the same constants \cite{Faber:1990ed}. 
\end{definition}
\tcbb{Here, we are interested in self-adjoint operators on finite dimensional spaces, but assume the operators 
$A$ and $B$ are discretizations on a sequence of meshes and that
the equivalence constants are independent of the mesh.}
More results on the equivalence of operators in a Hilbert space can be found in
\cite{Faber:1990ed}.

\begin{lemma}\label{lem:bound_block}
Consider the block matrix $\begin{pmatrix}~~A & -B \\ -C & ~~D\end{pmatrix}$. Suppose 
\begin{align*}
0 < a_0\|\mathbf{x}\| & \leq \|A\mathbf{x}\| \leq a_1\|\mathbf{x}\| , \hspace{4.5ex} \|B\mathbf{x}\| \leq b\|\mathbf{x}\|,\\
0 < d_0\|\mathbf{x}\| & \leq  \|D\mathbf{x}\| \leq d_1\|\mathbf{x}\| , \hspace{4ex} \|C\mathbf{x}\| \leq c\|\mathbf{x}\|,
\end{align*}
for all $\mathbf{x}$. Further, assume $a_0d_0 > bc$. Then,
\begin{align*}
0 < \eta_0 \leq \frac{\left\|\begin{pmatrix}~~A & -B \\ -C & ~~D\end{pmatrix}\begin{pmatrix}\mathbf{x}\\\mathbf{y}\end{pmatrix}\right\|^2 }{\|\mathbf{x}\|^2 + \|\mathbf{y}\|^2} \leq \eta_1,
\end{align*}
where
\begin{align*}
\eta_0 & = \frac{a_0^2 + b^2 + c^2 + d_0^2 - \sqrt{(a_0^2+b^2-c^2-d_0^2)^2 + 4(a_0c+bd_0)^2}}{2},\\
\eta_1 & = \frac{a_1^2 + b^2 + c^2 + d_1^2 + \sqrt{(a_1^2+b^2-c^2-d_1^2)^2 + 4(a_1c+bd_1)^2}}{2}.
\end{align*}
\end{lemma}
\begin{proof}
The proof is found in the Appendix.
\end{proof}

\subsection{Building a basis}\label{sec:2grid:basis}

\tcb{
In what follows, we denote $P$ and submatrices of $P$ represented in the basis of the right singular vectors with script letters. 
For example, $\mathcal{P} = V^*P$. Likewise, denote $\mathcal{R} = U^*R$. Note that
\begin{equation}\nonumber
P^*(QA)P = \mathcal{P}^*\Sigma\mathcal{P}  \quad \mbox{and}\quad
R^*(AQ)R = \mathcal{R}^*\Sigma\mathcal{R}.
\end{equation}
The transformed space allows for a natural separation of singular vectors with small singular values, which
need to be interpolated accurately, from singular vectors with larger singular values. While singular vectors with 
larger singular values need not be interpolated accurately, it will be shown below that $R$ and $P$ must have a 
similar action on corresponding left and right singular vectors. }

We begin the discussion by demonstrating that coarse-grid correction, $I - \Pi$, is invariant over any change of basis for $P$ and
$R$. If we let $B_P$ and $B_R$ be nonsingular $n_c\times n_c$ square matrices such that $\tilde{P} := PB_P$ and $\tilde{R} := RB_R$,
then, it is easy to show
\begin{align}
\Pi & =  P(R^*AP)^{-1}R^*A =  \tilde{P}(\tilde{R}^*A\tilde{P})^{-1}\tilde{R}^*A. \label{eq:cgc_basis}
\end{align}
Convergence of nonsymmetric AMG will be proved by developing appropriate bases for $R$ and $P$ under which to consider
convergence. \tcb{In particular, a representation of $R$ and $P$ is developed in terms of left and right singular vectors that is
fundamental to understanding convergence.}

\tcb{This section develops an appropriate basis for $P$. First, $P$ is expressed in a block column sense, $P = [W_1,W_2]$,
where $W_1$ and $W_2$ represent an $\ell^2$-orthogonal decomposition of $\mathbf{R}(P)$. In particular, $W_1$ is the 
$\ell^2$-orthogonal projection of the right singular vectors of $A$ with the smallest singular values onto $\mathbf{R}(P)$, 
and $W_2$  is the $\ell^2$-orthogonal complement of $W_1$ in $\mathbf{R}(P)$. A similar decomposition is developed 
for $\mathbf{R}(R)$.
Note that, in later sections, we start with this representation of $P$ to avoid introducing multiple change-of-basis matrices.}

Stability in the $QA$-norm is important to proving two-grid convergence, and can be analyzed in the $\ell^2$-norm in
the singular-vector-transformed space:
\begin{align}
\|\Pi\|_{QA}^2 & = \sup_{\mathbf{x}\neq\mathbf{0}} \frac{\langle QA \Pi\mathbf{x},\Pi \mathbf{x}\rangle}{\langle\mathbf{x},\mathbf{x}\rangle} 
= \sup_{\mathbf{y}\neq\mathbf{0}} \frac{\|\Sigma^{1/2}V^*P(R^*AP)^{-1}R^*U\Sigma\mathbf{y}\|^2}{\|\Sigma^{1/2}\mathbf{y}\|^2} \nonumber\\
& = \|\Sigma^{1/2}\mcP (\mcR^*\Sigma\mcP)^{-1}\mcR^*\Sigma^{1/2}\|^2. \label{eq:CGC_uv}
\end{align}
\
\tcb{Note that \eqref{eq:CGC_uv} holds for any change of bases.}

To prove two-grid and multilevel convergence using change-of-bases $\tilde{P}:= PB_P$ and $\tilde{R}:=RB_R$, several results
are needed:
\begin{enumerate}
\item \textit{Stability:} $\|\Pi\|_{QA} \leq C$,
\item \textit{Bounded change of bases:} $B_P^*B_P\sim_s I$ and $B_R^*B_R\sim_s I$,
\item \textit{Equivalence of inner products:} $(A_c^*A_c)^{1/2} \sim_s P^*QAP$.
\end{enumerate}
Stability is used to prove two-grid convergence, and is considered through a representation of $R$ and $P$ in
terms of singular vectors. Boundedness of the change-of-basis operators ensures that if $P$ is nicely bounded, $P^*P\sim_s I$,
then $\tilde{P}$ is also nicely bounded, $\tilde{P}^*\tilde{P}\sim_s I$. This is a subtle but important result for multilevel
convergence. The equivalence of inner products is also important for multilevel convergence and is discussed in Section \ref{sec:multigrid}.

Let $\Pi_{0_P} = P(P^*P)^{-1}P^*$ be the $\ell^2$-orthogonal projection onto the range of $P$, and define
\begin{align*}
V_1 &= \left[ \mathbf{v}_1, \mathbf{v}_2,\ldots \mathbf{v}_k\right], \qquad &\Sigma_1 &= \diag[\sigma_1,\sigma_2,\ldots,\sigma_k],
	\qquad &W_1 &= \Pi_{0_P}V_1,\\
V_2 &= \left[\mathbf{v}_{k+1}, \ldots, \mathbf{v}_n \right],  \qquad &\Sigma_2 &= \diag[\sigma_{k+1},\ldots,\sigma_n], 
	\qquad &N_1 &= (I - \Pi_{0_P})V_1,
\end{align*}
where \tcb{$k\leq n_c$} will be chosen later such that $\sigma_{k+1}\sim O(1)$. Let $W_2 = [\mathbf{w}_{k+1},\ldots,\mathbf{w}_{n_c}]$
be the $\ell^2$-orthogonal complement of $W_1$ in $\mathbf{R}(P)$, normalized so that \tcb{$W_2^* QA W_2 = I$}. There are many
choices for the basis of $W_2$. Below, a special basis will be constructed.

Assume that $P$ satisfies a \tcbb{FAP$(\beta,0)$} with respect to $QA$ with constant $K_{P}$. Choose
$k$ such that $ \delta_{P} := \sigma_k^{\beta} K_{P}^{1/2} <  1.0$. Note, smaller bounds on $\delta_P$ will be chosen
later for specific results on convergence. From the FAP$(\beta,0)$,
\begin{align}\nonumber
\| N_1 \mathbf{x} \|^2 &= \| (I-\Pi_{0_P})V_1\mathbf{x}\|^2 \leq K_{P} \langle (QA)^{2\beta} V_1\mathbf{x},V_1\mathbf{x}\rangle  \\ \label{eqn:normN}
&= K_{P} \langle \Sigma_1^{\beta}\mathbf{x},\Sigma_1^{\beta} \mathbf{x}\rangle \leq K_{P} \sigma_k^{2\beta} \|\mathbf{x}\|^2 = \delta_{P}^2 \|\mathbf{x}\|^2.
\end{align}
Because $\mathbf{R}(W_2)\subset\mathbf{R}(P)$, $\Pi_{0_P}W_2 = W_2$. By construction, $\mathbf{0} = W_1^*W_2 = V_1^*\Pi_{0_P}
W_2 = V_1^*W_2$, and $N_1^*W_2 = V_1^*(I - \Pi_{0_P})W_2 = \mathbf{0}$. Using this basis for $P$, we can write
\begin{align*}
\mathcal{P}= V^* P = [V_1, V_2]^* [W_1, W_2] = [V_1, V_2]^* [ V_1 - N_1, W_2 ]  
= \left[ \begin{array}{cc} I - \mcN_{11} & \mathbf{0} \\ ~~-\mcN_{21} & \mcW_{2} \end{array}\right],
\end{align*}
where 
\begin{align*}
\mcN_{11} & := V_1^*N_1, \hspace{3ex} \mcN_{21} := V_2^*N_1, \hspace{3ex} \mcW_2  := V_2^*W_2.
\end{align*}
Given $V_1V_1^* + V_2V_2^* = I$, it follows that $\mcW_2^*\mcN_{21} = -W_2^*(I - V_1V_1^*)N_1 = \mathbf{0}$. 

Noting the orthogonal decomposition $ \|V^*N_1\mathbf{x}\|^2 = \|V_1^*N_1\mathbf{x}\|^2 + \|V_2^*N_1\mathbf{x}\|^2
= \|\mcN_{11}\mathbf{x}\|^2 + \|\mcN_{21}\mathbf{x}\|^2$ and using \eqref{eqn:normN},
\begin{align*}
\|\mcN_{11}\mathbf{x}\|^2 + \|\mcN_{21}\mathbf{x}\|^2 = \|N_1\mathbf{x}\|^2 \leq K_{P} \| \Sigma_1^{\beta}\mathbf{x}\|^2,
\end{align*}
and, for some $\theta_{\mathbf{x}}$,
\begin{equation}\label{eq:fap_basis}
\|{\mcN}_{11}\mathbf{x}\|^2 \leq \cos^2(\theta_{\mathbf{x}})K_{P}\|\Sigma_1^{\beta}\mathbf{x}\|^2, \quad
	\|{\mcN}_{21}\mathbf{x}\|^2 \leq \sin^2(\theta_{\mathbf{x}}) K_{P}\|\Sigma_1^{\beta}\mathbf{x}\|^2.
\end{equation}
In the development below, we will replace $\mathbf{x}$ in \eqref{eq:fap_basis} with $\Sigma_1^{-\beta}\mathbf{x}$. 

By assumption of a FAP$(\beta,0)$ and an appropriate choice of $k$, $\|\mcN_{11}\| \leq \|N_1\| \leq \delta_P < 1.$ 
\tcb{Then, $\|(I - \mcN_{11})\mathbf{x}\| > (1-\delta_P)\|\mathbf{x}\|$ for all $\mathbf{x}$, implying $(I-\mcN_{11})$ is nonsingular and invertible.}
Consider a further change of basis to obtain
\begin{equation}\label{eq:basis1}
\tilde{P} = P\left[\begin{array}{cc} (I-\mcN_{11})^{-1} & \mathbf{0} \\ \mathbf{0} & I \end{array}\right] = V  \left[ \begin{array}{cc} I 
	& \mathbf{0} \\ - \mcN_{21}(I - \mcN_{11})^{-1}\Sigma_1^{-\beta}\Sigma_1^{\beta} & \mcW_{2} \end{array}\right].
\end{equation}
Here, we denote $\widehat{\mcN}_2 = \mcN_{21}(I - \mcN_{11})^{-1}\Sigma_1^{-\beta}$, and $\tilde{P}$ takes the form
\begin{equation}\label{eq:basis2}
\tilde{P} = V  \left[ \begin{array}{cc} I & \mathbf{0} \\ - \widehat{\mcN}_{2}\Sigma_1^{\beta} & \mcW_{2} \end{array}\right] .
\end{equation}
It is reasonable to take pause and ask why we added a factor of $\Sigma_1^{-\beta}$ to the block $\mcN_{21}(I - \mcN_{11})^{-1}$ in
\eqref{eq:basis1}. As a result of the FAP$(\beta,0)$, it can be shown that $\mcN_{21}(I - \mcN_{11})^{-1}\Sigma_1^{-q}$ is nicely bounded for powers
of $q\leq \beta$. In particular, we can write $\widehat{\mcN_2} = \mcN_{21}\Sigma_1^{-\beta} (I- \Sigma_1^{\beta}\mcN_{11}\Sigma_1^{-\beta})^{-1}$.
Note that, from \eqref{eq:fap_basis}, $\|\Sigma_1^\beta\mcN_{11} \Sigma_1^{-\beta}\| \leq \|\Sigma_1^{\beta}\|\|\mcN_{11}\Sigma_1^{-\beta}\| \leq 
\sigma_k^\beta K_{P}^{1/2} = \delta_P < 1$, and, thus, $I - \Sigma_1^\beta{\mcN}_{11}\Sigma_1^{-\beta}$ is
invertible. Again using \eqref{eq:fap_basis},
{
\begin{align*}
\|\widehat{\mcN}_{2}\| & = \sup_{\mathbf{x}\neq\mathbf{0}} \frac{\|\mcN_{21}\Sigma_1^{-\beta}(I - \Sigma_1^\beta{\mcN}_{11}\Sigma_1^{-\beta})^{-1}
	\mathbf{x}\|}{\|\mathbf{x}\|} 
=  \sup_{\mathbf{y}\neq\mathbf{0}} \frac{\|\mcN_{21}\Sigma_1^{-\beta}\mathbf{y}\|}
	{\|(I - \Sigma_1^\beta{\mcN}_{11}\Sigma_1^{-\beta})\mathbf{y}\|} \\
& =  \sup_{\mathbf{y}\neq\mathbf{0}} \frac{\|\mcN_{21}\Sigma_1^{-\beta}\mathbf{y}\|}
	{\|\mathbf{y}\| - \|\Sigma_1^\beta{\mcN}_{11}\Sigma_1^{-\beta}\mathbf{y}\|}
	\leq \sup_{\mathbf{y}\neq\mathbf{0}} \frac{\sin(\theta_{\mathbf{y}})K_{P}^{1/2}}{1-\delta_P\cos{\theta_{\mathbf{y}}}},
\end{align*}}
\hspace{-1ex}where, recall, $\delta_P := \sigma_k^{\beta}K_{P}^{1/2}$. \tcb{The maximum over $\mathbf{y}$
occurs when $\cos(\theta_{\mathbf{y}}) = \delta_P$, leading to the bound}
\begin{align*}
\|\widehat{\mcN}_{2}\|^2 \leq \frac{K_{P}}{1-\delta_P^2} := \widehat{K}_{P}.
\end{align*}
The significance of this result is that the block in \eqref{eq:basis2}, $\widehat{\mcN}_2\Sigma_1^{\beta}$, is now bounded when
multiplied by $\Sigma_1^{-\ell}$ for $\ell \leq \beta$. In particular,
\begin{align*}
\|\widehat{\mcN}_{2}\Sigma_1^{\beta-\ell}\| \leq \sigma_k^{\beta-\ell} \widehat{K}_{P}^{1/2} \leq  \widehat{K}_{P}^{1/2} ,
\end{align*}
for $\ell \leq\beta$. \tcb{Note, this is a stronger result in terms of bounding blocks in \eqref{eq:basis1} than can be
obtained through the more natural submultiplicative bound of $\|\mcN_{21}(I - \mcN_{11})^{-1}\|$ based on
$\|\mcN_{21}\|$ and $\|(I - \mcN_{11})^{-1}\|$.} Such a result highlights the significance of the
order of FAP satisfied by $P$, and is important in proving stability and coarse-grid equivalence.

The preceding discussion developed a representation of $P$ in terms the right singular vectors of $A$. An equivalent
approach can be used to develop a representation of $R$ in terms of the left singular vectors
of $A$, and results are summarized in the following lemma.

\begin{lemma}[Bases for $R$ and $P$]\label{lem:PRdecomp}
Assume that $P$ satisfies a FAP$(\beta,0)$ with respect to $QA$, with $\beta > 0$ and  constant $K_{P}$, and that 
$R$ satisfies a FAP$(\gamma,0)$ with respect to $AQ$, with  $\gamma > 0$ and constant $K_{R}$. Further, assume
that $P^*P \sim_s I \sim_s R^*R$. Choose $k\leq n_c$ such that
$\delta_{P} := \sigma_k^{\beta}K_{P}^{1/2}  < 1/\sqrt{2}$ and   $\delta_{R} := \sigma_k^{\gamma} K_{R}^{1/2} < 1/\sqrt{2}$. 
Then, there exist bases, $B_P$ for $P$ and $B_R$ for $R$, such that, \tcb{if $k<n_c$,}
\begin{align*}
\tilde{P} = PB_P & = V\tilde{\mathcal{P}} = \left[ V_1, V_2\right]\left[\begin{array}{cc} I_{k} & \mathbf{0} \\ -\widehat{\mcN}_2\Sigma_1^{\beta} & \mcW_2 \end{array}\right],\\
\tilde{R} = RB_R & = U \tilde{\mathcal{R}} = \left[ U_1, U_2\right]\left[\begin{array}{cc} I_{k} & \mathbf{0} \\ -\widehat{\mcM}_2\Sigma_1^{\gamma} & \mcZ_2 \end{array}\right],
\end{align*}
where
\begin{enumerate}
\item $\mcW_2^*\Sigma_2\mcW_2 = \mcZ_2^*\Sigma_2\mcZ_2 = I$, 
\item $\mcW_2^*\widehat{\mcN}_2 = \mcZ_2^*\widehat{\mcM}_2 = \mathbf{0}$,
\item $\|\widehat{\mcN}_2\| \leq \widehat{K}_{P}^{1/2} :=  \left(\frac{K_{P}}{1-\delta_{P}^2}\right)^{1/2}$,
and $\|\widehat{\mcM}_2\| \leq \widehat{K}_{R}^{1/2} :=  \left(\frac{K_{R}}{1-\delta_{R}^2}\right)^{1/2}$,
\item $B_P^*B_P \sim_s I \sim_s B_R^*B_R$.
\end{enumerate}
Furthermore, \tcb{the bases $B_P$ and $B_R$  can be chosen such that}
\begin{equation*}
\mcZ_2^*\Sigma_2\mcW_2 = S_2 = \diag{[ s_1, s_2,\ldots, s_{n_c-k}]} ,
\end{equation*} 
with $0 \leq  s_1  \leq \ldots \leq s_{n_c-k} \leq 1$. These singular values are the cosines of the angles
between the subspaces $W_2$ and $QZ_2$ in the $QA$ inner product.

\tcb{If $k=n_c$, then $\mcW_2$ and $\mcZ_2$ are empty and conclusions $3$ and $4$ above hold.}
\end{lemma}
\begin{proof}
\tcb{When $k<n_c$}, results (1), (2), and (3) follow from the discussion above. It remains to show that $B_P^*B_P \sim_s I \sim_s B_R^*B_R$. 
This is accomplished by observing that, by construction,
\begin{equation*}
\tilde{P}^*\tilde{P} = \left[\begin{array}{cc} I+\Sigma_1^{\beta}\widehat{\mcN}_2^*\widehat{\mcN}_2\Sigma_1^{\beta} & \mathbf{0} \\ 
\mathbf{0} & \mcW_2^*\mcW_2 \end{array}\right].
\end{equation*}
By assumption, $\delta_{P}< 1/\sqrt{2}$, and
\begin{equation*}
\|\Sigma_1^{\beta}\widehat{\mcN}_2^*\widehat{\mcN}_2\Sigma_1^{\beta}\| = \|\widehat{\mcN}_2\Sigma_1^{\beta}\|^2 \leq
\|\widehat{\mcN}_2\|^2\|\Sigma_1^{\beta}\|^2 \leq \widehat{K}_{P} \sigma_k^{2\beta} = \frac{\delta_{P}^2}{1-\delta_{P}^2} < 1.
\end{equation*}
This implies that 
\begin{equation}\label{eq:P*P}
\langle \mathbf{x}, \mathbf{x}\rangle \leq \langle (I+\Sigma_1^{\beta}\widehat{\mcN}_2^*\widehat{\mcN}_2\Sigma_1^{\beta})
	\mathbf{x},\mathbf{x}\rangle \leq 2\langle \mathbf{x}, \mathbf{x} \rangle.
\end{equation}
Also, 
\begin{equation*}
1 = \frac{\langle\Sigma_2 \mcW_2 \mathbf{x}, \mcW_2 \mathbf{x} \rangle}{\langle \mathbf{x},\mathbf{x} \rangle}
 \leq
\frac{\langle \mcW_2 \mathbf{x}, \mcW_2 \mathbf{x} \rangle}{\langle \mathbf{x},\mathbf{x} \rangle}
\leq \frac{1}{\sigma_{k+1} }\frac{\langle\Sigma_2 \mcW_2 \mathbf{x}, \mcW_2 \mathbf{x} \rangle}{\langle \mathbf{x},\mathbf{x} \rangle}
=\frac{1}{\sigma_{k+1}}.
\end{equation*}
Assume that $\sigma_k$ is chosen as large as possible but still satisfies the hypotheses. Then,
\begin{equation*}
\frac{1}{\sigma_{k+1} }\leq \max [(2K_{P})^{1/2\beta},(2K_{R})^{1/2\gamma}] .
\end{equation*} 
Thus, $\| \mathbf{x}\|^2 \leq \langle \tilde{P} \mathbf{x}, \tilde{P} \mathbf{x}\rangle \leq \max\{2, 1/\sigma_{k+1}\} \| \mathbf{x}\|^2$, which implies
$\tilde{P}^*\tilde{P} \sim_s I$. Together with the assumption $P^*P \sim_s I$, this implies $B_P^*B_P \sim_s I$. 
A similar result proves $B_R^*B_R \sim_s I$.
\tcb{When $k = n_c$, $\mcW_2$ is empty and \eqref{eq:P*P} yields $\tilde{P}^*\tilde{P} \sim_s I$. By the argument above $B_P^*B_P\sim_s I$.
A similar argument yields $B_R^*B_R \sim_s I$}

\tcb{To complete the proof, again assume $k<n_c$} and let
\begin{equation*}
\mcZ_2^*\Sigma_2\mcW_2 = \widehat{U}_2S_2\widehat{V}_2^*,
\end{equation*}
be a SVD. \tcb{Recall we are free to choose any bases for $\mcW_2$ and $\mcZ_2$. Consider the change of bases  in which $\mcW_2\mapsfrom\mcW_2\widehat{V}_2$ and $\mcZ_2\mapsfrom\mcZ_2\widehat{U}_2$. 
In these bases,  $\mcZ_2^*\Sigma_2\mcW_2 = S_2$. Moreover, in these bases,
$({\Sigma_2^{1/2}}\mcW_2)$ and $({\Sigma_2^{1/2}}\mcZ_2)$ remain orthonormal, 
yielding the bounds 
$ 0  \leq s_1 \leq s_{n_c-k} \leq 1$.}
$(\beta,0)$
To verify the last statement in the theorem, recall $\mcW_2 = V_2^*W_2$ and $\mcZ_2 = U_2^*Z_2$ and, by construction, $V_1^*W_2 = \mathbf{0}$.
Then,
\begin{align*}
\langle \Sigma_2\mcW_2 \mathbf{x},\mcZ_2\mathbf{y} \rangle &= \langle \Sigma_2V_2^*W_2\mathbf{x},U_2^*Z_2\mathbf{y}\rangle \\
& = \langle U_2\Sigma_2V_2^*W_2\mathbf{x}, Z_2\mathbf{y}\rangle \\
& = \langle (U_1\Sigma_1V_1^* + U_2\Sigma_2V_2^*) W_2\mathbf{x},Z_2\mathbf{y}\rangle \\
& = \langle AW_2\mathbf{x}, Z_2\mathbf{y}\rangle.
\end{align*}
Noting that the singular values are stationary values of the following quotients \cite{VanLoan:1976hl}, it follows that
{\small
\begin{align*}
\frac{\langle \Sigma_2 \mcW_2 \mathbf{x}, \mcZ_2 \mathbf{y} \rangle}{\|\mathbf{x}\|\|\mathbf{y}\| } =
	\frac{\langle AW_2 \mathbf{x}, Z_2 \mathbf{y} \rangle}{\|W_2 \mathbf{x}\|_{QA}\|Z_2 \mathbf{y}\|_{AQ} }  =
	\frac{\langle QAW_2 \mathbf{x}, QZ_2 \mathbf{y} \rangle}{\|W_2 \mathbf{x}\|_{QA}\|QZ_2 \mathbf{y}\|_{QA} } = 
		\frac{\langle W_2 \mathbf{x}, QZ_2 \mathbf{y} \rangle_{QA}}{\|W_2 \mathbf{x}\|_{QA}\|QZ_2 \mathbf{y}\|_{QA} }.
\end{align*} 
}Equivalently, this defines the cosines of angles between $\mathcal{R}(W_2)$ and $\mathcal{R}(QZ_2)$ in the $QA$-inner product. 
\end{proof}

\begin{remark}\label{rem:angle}
Here, $s_1 = \cos \theta_{max}$, 
where $\theta_{max}$ is the maximum angle between subspaces $\mathbf{R}(W_2)$ and
$\mathbf{R}(QZ_2)$ in the $QA$-inner product. If $\mathbf{R}(W_2) = \mathbf{R}(QZ_2)$, then $\theta_{max} = 0$ and $s_1 = 1$.
The less the spaces overlap, that is, the larger the opening angle between the spaces, the smaller $s_1$ will be.
\end{remark}

\subsection{Stability of $\Pi$ and two-grid convergence}\label{sec:2grid:stability}

We are now in position to prove stability of $\Pi$ under appropriate hypotheses. Sufficient conditions include FAPs on $R$
and $P$, as well as an additional hypothesis relating the behavior of $R$ and $P$ on the singular vectors associated with larger
singular values. 

\begin{theorem}[Stability]\label{th:stability}
Assume that $P^*P\sim_s I$, and $P$ satisfies a FAP$(\beta,0)$ with respect to $QA$, with $\beta \geq 1/2$ and constant $K_{P}$. Similarly,
assume that $R^*R\sim_s I$, and $R$ satisfies a FAP$(\gamma,0)$ with respect to $AQ$, with $\gamma \geq 1/2$ constant $K_{R}$,
where $\beta+ \gamma > 1$. Assume there exists  $k \leq n_c$ such that:
\begin{enumerate}
\item $\delta_{P} := \sigma_k^{\beta}K_{P}^{1/2}  < 1/\sqrt{2}$,  \qquad (Denote $\widehat{K}_{P} := K_{P}/(1-\delta_{P}^2)$ )
\item $\delta_R := \sigma_k^{\gamma }K_{R}^{1/2}  < 1/\sqrt{2}$,  \qquad (Denote $\widehat{K}_{R} := K_{R}/(1-\delta_{R}^2)$ )
\item $ \delta_{PR}^2:=\sigma_k^{\beta+\gamma -1} \widehat{K}_{P}^{1/2}\widehat{K}_{R}^{1/2}  < 1/2$.
\end{enumerate}
\tcb{Finally, if $k < n_c$, assume that}
\begin{equation}\label{eq:ass}
s_1 > \frac{\delta_{PR}^2}{(1-\delta_{PR}^2)}. \end{equation}
Then, $\|\Pi\|_{QA}^2  \leq C_{\Pi}$. A precise bound for $C_{\Pi}$ appears in (\ref{eqn:Cbound}) in the proof.
\end{theorem}

\begin{proof} 
First note that the assumptions here satisfy those of Lemma \ref{lem:PRdecomp}. \tcb{The proof of the case in which $k=n_c$ is a simplification 
of the proof for the case $k<n_c$ and will be omitted. Assume $k<n_c$.}
\tcb{
Using the fact that $\|\Pi\|_{QA}$ is invariant to a change of basis, appealing to \eqref{eq:CGC_uv}, and using the decomposition of $R$ and $P$ developed in Lemma \ref{lem:PRdecomp}, we have
}
\begin{align} \nonumber
\| \Pi\|_{QA}^2 &= \| \Sigma^{1/2}\mcP (\mcR^*\Sigma\mcP)^{-1} \mcR^*\Sigma^{1/2} \|^2  \\\nonumber
& = \| \Sigma^{1/2}\tilde{\mcP} (\tilde{\mcR}^*\Sigma \tilde{\mcP})^{-1} \tilde{\mcR}^*\Sigma^{1/2} \|^2 \\ \label{eqn:first_term}&= \left\| \left[\begin{array}{cc} I & \mathbf{0} \\ -\Sigma_2^{1/2}\widehat{\mcN}_2\Sigma_1^{\beta-1/2} & \Sigma_2^{1/2}\mcW_2 \end{array}\right] \right. \\
 \label{eqn:middle_term}&\qquad\qquad\left[ \begin{array}{cc} I+\Sigma_1^{\gamma-1/2} \widehat{\mcM}_2^*\Sigma_2\widehat{\mcN}_2 \Sigma_1^{\beta-1/2}
& -\Sigma_1^{\gamma-1/2}\widehat{\mcM}_2^*\Sigma_2\mcW_2 \\ 
-\mcZ_2^*\Sigma_2\widehat{\mcN}_2\Sigma_1^{\beta-1/2} & S_2 \end{array}\right]^{-1} \\ \label{eqn:third_term}&\left.\qquad\qquad\qquad\qquad  \left[\begin{array}{cc} I & -\Sigma_1^{\gamma-1/2} \widehat{\mcM}_2^*\Sigma_2^{1/2} \\ \mathbf{0} & \mcZ_2^*\Sigma_2^{1/2}\end{array}\right]  \right\|^2 .
\end{align}
We will bound each of these three $2\times 2$ block matrices using Lemma \ref{lem:bound_block}. Nonzero off-diagonal blocks must
be bounded from above in each case, which can be done using Lemma \ref{lem:PRdecomp}, the orthonormality of
$\mcW_2^*{\Sigma_2}\mcW_2 = \mcZ_2^*{\Sigma_2}\mcZ_2 = I$, and the scaling of $A$ such that $\sigma_i\leq 1$ for all $i$:
\begin{align*}
\| \mcZ_2^*\Sigma_2\widehat{\mcN}_2\Sigma_1^{\beta-1/2} \| &\leq \|\Sigma_2^{1/2}\widehat{\mcN}_2\Sigma_1^{\beta-1/2}\|
	\leq \sigma_k^{\beta-1/2}\|\widehat{\mcN}_2\| \leq \sigma_k^{\beta-1/2}\widehat{K}_{P}^{1/2}, \\
\| \Sigma_1^{\gamma-1/2}\widehat{\mcM}_2^*\Sigma_2\mcW_2  \| &\leq \| \Sigma_1^{\gamma-1/2} \widehat{\mcM}_2^*\Sigma_2^{1/2}\|
	\leq \sigma_k^{\gamma-1/2}\|\widehat{\mcM}_2\| \leq \sigma_k^{\gamma-1/2}\widehat{K}_{R}^{1/2}.
\end{align*}
Note that this is where the assumption of $\gamma,\beta \geq 1/2$ is important.
Both diagonal blocks of the first term \eqref{eqn:first_term} and third term \eqref{eqn:third_term} are bounded above and below
by one; the upper diagonal block in each case is the identity, and the lower diagonal blocks are given by $\|\Sigma_2^{1/2}\mcW_2\| =
\|\mcZ_2^*\Sigma_2^{1/2}\| = 1$. Diagonal blocks in the middle term can be bounded in a similar manner, noting that
\begin{align*}
(1- \delta_{PR}^2)\|\mathbf{x}\| &\leq \left\| (I + \Sigma_1^{\gamma-1/2} \widehat{\mcM}_2^*\Sigma_2\widehat{\mcN}_2 \Sigma_1^{\beta-1/2})\mathbf{x}\right\|
	\leq (1+ \delta_{PR}^2)\|\mathbf{x}\|, \\
s_1 \|\mathbf{x}\| &\leq \| S_2\mathbf{x} \| \leq \| \mathbf{x} \|. 
\end{align*}

Then, the first term (\ref{eqn:first_term}) and third term (\ref{eqn:third_term}) are easily bounded above:
{\small
\begin{align*}
\left\| \left[\begin{array}{cc} I & 0 \\ -\Sigma_2^{1/2}\widehat{\mcN}_2\Sigma_1^{\beta-1/2} & \Sigma_2^{1/2}\mcW_2 \end{array}\right] \right\|^2
& \leq 1 + \frac{\sigma_k^{2\beta-1}\widehat{K}_{P} + \sqrt{\sigma_k^{4\beta-2}\widehat{K}_{P}^2 + 4\sigma_k^{2\beta-1}\widehat{K}_{P}}}{2} \\
& < 2 + \sigma_k^{2\beta-1}\widehat{K}_{P}, \\
 \left\|\left[\begin{array}{cc} I & -\Sigma_1^{\gamma-1/2} \widehat{\mcM}_2^*\Sigma_2^{1/2} \\ 0 & \mcZ_2^*\Sigma_2^{1/2}\end{array}\right]  \right\|^2
& \leq 1 + \frac{\sigma_k^{2\gamma-1}\widehat{K}_{R} + \sqrt{\sigma_k^{4\gamma-2}\widehat{K}_{R}^2 + 4\sigma_k^{2\gamma-1}\widehat{K}_{R}}}{2} \\
& < 2 + \sigma_k^{2\gamma-1}\widehat{K}_{R}. 
\end{align*}
}To bound the middle term \eqref{eqn:middle_term} from above, note that if $\eta_0\|\mathbf{\mathbf{x}}\|^2 \leq \|A\mathbf{\mathbf{x}}\|^2$, then
\begin{align*}
\|A^{-1}\| = \sup_{\mathbf{\mathbf{x}}\neq\mathbf{0}} \frac{\|A^{-1}\mathbf{\mathbf{x}}\|^2}{\|\mathbf{\mathbf{x}}\|^2} 
	= \sup_{\mathbf{y}\neq\mathbf{0}} \frac{\|\mathbf{y}\|^2}{\|A\mathbf{y}\|^2} \leq \frac{1}{\eta_0}.
\end{align*}
In notation of Lemma \ref{lem:bound_block}, blocks of the middle term have bounds 
\begin{equation*}
a_0 = 1-\delta_{PR}^2, \hspace{4ex} d_0 = s_1,\hspace{4ex} b = \sigma_k^{(\gamma-1/2)}\widehat{K}_{R}^{1/2}, \hspace{4ex} c  = \sigma_k^{(\beta-1/2)}\widehat{K}_{P}^{1/2}.
\end{equation*}
Lemma \ref{lem:bound_block} applies when $a_0d_0 > bc$. Plugging in, this constraint is satisfied when
\begin{equation}\label{eqn:hyp4}
1 \geq s_1 > \frac{\delta_{PR}^2}{1-\delta_{PR}^2 } .
\end{equation}
Equation (\ref{eqn:hyp4}) is the final assumption above, which requires $\delta_{PR}^2 < 1/2$ (Assumption 3), 
which can only be guaranteed if
$\beta+\gamma > 1$. Lemma \ref{lem:bound_block} then yields
\begin{align*}
\eta_0 & = \frac{a_0^2 + b^2 + c^2 + d_0^2 - \sqrt{(a_0^2+b^2-c^2-d_0^2)^2 + 4(a_0c+bd_0)^2}}{2} > 0,\\
\end{align*}

Putting this all together yields
\begin{equation}\label{eqn:Cbound}
\| \Pi\|_{QA}^2 \leq \frac{(1+ \sigma_k^{2\beta-1} \widehat{K}_{P})(1+ \sigma_k^{2\gamma-1} \widehat{K}_{R})}{\eta_0}.
\end{equation}
\end{proof}

Equation \eqref{eqn:Cbound} provides clear separation of three measures of an AMG hierarchy: the two terms in the numerator
reflect the approximation properties on $R$ and $P$, and the size of the denominator reflects the relation of the action of
$R$ and $P$ on singular vectors associated with larger singular values. Note that the approximation properties of
$R$ and $P$ do not have to be equal. This proof of stability requires at least a FAP$(1/2,0)$ (WAP) on each, and together, hypotheses
require the slightly stronger statement, $\beta+\gamma > 1$.\footnote{A similar derivation under the stronger initial assumption that
$V_1^*P$ and $U_1^*R$ are nonsingular for $k=n_c$ leads to stability, with similar assumptions on the action of $R$ and $P$ on singular vectors
associated with larger singular values, and $\beta+\gamma = 1$; that is, $R$ and $P$ both satisfy a FAP$(1/2,0)$. The stronger requirement
in Theorem \ref{th:stability}, $\beta+\gamma > 1$, may be a shortcoming of this line of proof.} Beyond satisfying a FAP$(1/2,0)$, stronger
approximation properties of $P$ or $R$ are reflected through larger $\beta$ and $\gamma$, both of which reduce the
bound on $\|\Pi\|_{QA}$. 

For larger singular values, approximation properties hold trivially and, for SPD matrices \tcb{with $R=P$}, this means that one need only
pay attention to singular vectors with small singular values. \tcb{In the nonsymmetric setting, stability requires the additional
final assumption in Theorem \ref{th:stability} \eqref{eq:ass},
which establishes a relationship between $\delta_{PR}$ and $s_1$.} This hypothesis is derived
from relating the action of $R$ and $P$ on singular vectors associated with larger singular values. From Lemma \ref{lem:PRdecomp}, we know that
$\mcZ_2^*\Sigma_2\mcW_2 = \text{diag}[s_1,...,s_{n_c-k}]$, where these values are the cosines of angles between subspaces
$W_2$ and $QZ_2$. For example, suppose the $j$th right singular vector $\mathbf{v}_j\subset \mathbf{R}(P)$ for $j > k$,
but the $j$th left singular vector $\mathbf{u}_j{\perp} \mathbf{R}(R)$. Then there exists a vector $\mathbf{x}$ such that
$\langle W_2\mathbf{x},QZ_2\mathbf{y}\rangle = \langle UV^*\mathbf{v}_j, Z_2\mathbf{y}\rangle = \langle \mathbf{u}_j, Z_2\mathbf{y}\rangle = 0$
for all $\mathbf{y}$. Then, $\theta_{max} = \pi/2$, $s_1 = 0$, and we do not have stability (see Remark \ref{rem:angle} and
Example \ref{ex:counter}). Thus, $R$ and
$P$ must have a similar action on left and right singular vectors associated with large singular values, respectively. How strong the
constraint is depends on approximation properties. When $\delta_{PR}  = 0$, the constraint is $s_1 > 0$; that is, $S_2$ need
only be nonsingular. When $\delta_{PR} \geq 1/\sqrt{2}$, the restriction is $s_1 > 1$ which is not possible to satisfy. By choosing a
smaller $k$, $\delta_{PR}$ can be made smaller. However, choosing $k$ smaller also makes the dimension of spaces
$W_2$ and $Z_2$ larger,
which makes $s_1$ smaller, and less likely to satisfy the constraint. The hypotheses hold only if there is some $k$ 
with $\sigma_k \sim O(1)$ for which all the hypotheses, including 
\eqref{eq:ass} holds. 

Stronger approximation properties, through either smaller constants, $K_{P}$ and $K_{R}$, or larger $\beta$ and $\gamma$,
make $\delta_{PR}$ smaller. This makes it easier to satisfy the hypotheses of Theorem \ref{th:stability} and, in particular, the
constraint on $s_1$. It is also worth considering how accurate approximation properties must be. Suppose we assume equal
approximation properties on $R$ and $P$ with power $\beta = \delta = \delta_P = \delta_R$ and $K_F := K_{P} = K_{R}$.
Then, bounding $\delta_{PR}^2 < 1/2$ is equivalent to 
\begin{align*}
K_F <\frac{1}{\sigma_k^{2\beta-1}} \frac{1}{2+ \sigma_k}.
\end{align*}
Under this line of proof, more accurate approximation (smaller $K_F$) is required for weaker approximation properties (smaller $\beta$),
while stronger approximation properties (larger $\beta$) can tolerate a larger $K_F$. Large $\sigma_k$ also requires a more accurate
approximation through smaller $K_F$.

{\color{black}
\begin{remark}
Theorem \ref{th:stability} is stated and proven in full generality. It is important to note that this
same line of proof is viable, and yields the correct result, in the limit as the system becomes SPD.
More generally, consider the case in which $R = Q^*P$.  As mentioned above, this yields $\Pi$ as the
$QA$-orthogonal projection onto $\mathbf{R}(P)$ and $\|\Pi\|_{QA} = 1$.  
If $A$ is SPD, it becomes a special case in which $Q = I$ and $R=P$.
Consider the bases constructed in Lemma \ref{lem:PRdecomp}. The condition $\tilde{R}=Q^*\tilde{P}$ is equivalent to
$\tmcR = U^*\tilde{R} = V^* \tilde{P} = \tmcP$. 
Moreover, assume $\| U^*R-V^*P\|_{QA} = \|\mcR-\mcP\|_{\Sigma} \leq \epsilon$. The proof of Theorem \ref{th:stability}
can be used to show that the limit as $\epsilon \rightarrow 0$ yields $\|\Pi\|_{QA} = 1.0$.
\end{remark}
}

\section{Multilevel convergence}
\label{sec:multigrid}

Recall from \eqref{eq:cgc_basis} that coarse-grid correction is invariant under a change of basis, $\tilde{P} = {P}B_P$ and
$\tilde{R} = {R}B_R$, for change of basis matrices $B_P$ and $B_R$. Here, we use the bases developed in Lemma \ref{lem:PRdecomp}
to consider multilevel convergence in the nonsymmetric setting. There are two approximations that must be accounted for in considering
multilevel error propagation of coarse-grid correction, which do not arise in the two-level setting. First, and consistent with SPD multigrid theory,
we must account for an inexact coarse-grid solve given by recursively calling AMG on the coarse-grid problem. The nonsymmetric setting
poses additional difficulties in this recursive call. Specifically, some correction is interpolated to the fine grid, which assumes an
inner-product form along the lines of:
\begin{align*} 
\langle P\mathcal{V}_c\mathbf{e}_c, P\mathcal{V}_c\mathbf{e}_c\rangle_{QA} = \langle \mathcal{V}_c\mathbf{e}_c,\mathcal{V}_c\mathbf{e}_c,\rangle_{P^*QAP},
\end{align*}
where $\mathcal{V}_c$ is the error-propagation operator of the approximate coarse-grid solve. For SPD matrices, $P^*QAP =P^*AP = A_c$,
which is exactly the coarse-grid operator formed in practice, on which a recursive assumption is made, $\|\mathcal{V}_c\|_{P^*QAP} < 1$.
In the nonsymmetric setting, the coarse-grid operator is defined as $A_c:=R^*AP$, and the corresponding $Q_cA_c := (A_c^*A_c)^{\frac{1}{2}}$-norm
that we are studying is no longer equal to $P^*QAP$. Then, the
recursive assumption of coarse-grid convergence is with respect to the $\sqrt{A_c^*A_c}$-norm, as opposed to the $P^*QAP$-norm. Thus,
a fundamental piece of proving multilevel AMG convergence in the nonsymmetric setting is  to prove an equivalence between inner products 
${P}^*QA{P}$ and $(A_c^*A_c)^{\frac{1}{2}}$.

Conditions for equivalence between inner products are established in Section \ref{sec:multigrid:equiv}. Section \ref{sec:multigrid:conv}
then combines all of the pieces developed so far and \tcbb{establishes sufficient conditions for $W$-cycle convergence of AMG in the nonsymmetric setting. }

\tcbb{
Notationally, let the hierarchy consist of $L$ levels, where the original operator is denoted $A=A_1$ and the sequence of transfer operators
by $P_\ell, R_\ell$, for $\ell = 1,\ldots, L$. These
yield the sequence of coarse grid operators, $A_{\ell+1} = R_\ell^*A_\ell P_\ell$ of dimension $n_{\ell+1}$. 
Assume that $P_\ell$ and $R_\ell$ are chosen so that $\|A_{\ell+1} \|= 1$. Denote the singular values of $A_\ell$ by
 $0< \sigma_{1}^\ell \leq \sigma_{2}^\ell \leq\cdots\leq \sigma_{n_\ell}^\ell = 1$. 
 Assume that the next coarser level is chosen sufficiently large, $n_{\ell+1} < n_\ell$, such that $C_\sigma \leq \sigma_{n_{\ell+1}}^\ell \leq 1$ 
 where $C_\sigma \sim O(1)$, independent of grid level (when the meaning  is clear, the superscripts $\ell$ will be omitted).
}

\subsection{Equivalence of inner products}\label{sec:multigrid:equiv}

Proving the necessary equivalence of inner products will be accomplished by proving a stronger statement, the norm equivalence
of $A_c:= R^*AP$ and $P^*QAP$. Notice that 
\begin{align}\label{eq:equivIP1}
\frac{\| A_c \mathbf{x} \|^2}{\| P^*(QA) P\mathbf{x} \|^2} 
	= \frac{\| (A_c^*A_c)^{1/2} \mathbf{x}\|^2}{\| P^*(QA) P\mathbf{x} \|^2},
\end{align}
that is, $A_c\sim_n P^*QAP$ is equivalent to $(A_c^*A_c)^{\frac{1}{2}}\sim_n P^*QAP$.
Given that $(A_c^*A_c)^{\frac{1}{2}}$ and ${P}^*QA{P}$ are both self-adjoint, norm equivalence
implies spectral equivalence, with the same constants \cite{Faber:1990ed}. Spectral equivalence,
$(A_c^*A_c)^{\frac{1}{2}}\sim_s P^*QAP$, then gives bounds used in the proof of multilevel convergence:
\begin{align}
c_0 \leq \frac{\|\mathbf{y}_c\|^2_{(A_c^*A_c)^{\frac{1}{2}}}}{\|{P}\mathbf{y}_c\|^2_{QA}\hspace{2.5ex}} \leq 
	c_1,\label{eq:equiv9}
\end{align}
for some constants, $c_0$ and $c_1$, and all coarse-grid vectors $\mathbf{y}_c$. 

Norm equivalence is proven in Lemma \ref{lem:equivIP}. Conditions are consistent with those sufficient for stability
(Theorem \ref{th:stability}), with an additional, stronger approximation property assumed on $P$: $\beta \geq 1$.
That is, for this result $P$ must satisfy a SSAP or, equivalently, a SAP. \tcb{On the other hand, the basic requirement on
$R$ is only $\gamma > 0$. Of course, larger $\gamma$, that is, better approximation properties of $R$, make satisfying
the other hypotheses easier. Moreover, if $R=Q^*P$, then the equivalence is immediate.}

\begin{lemma}[Equivalence of Inner Products]\label{lem:equivIP}
Assume that $P^*P\sim_s I$, and $P$ satisfies a FAP$(\beta,0)$ with respect to $QA$, with $\beta \geq 1$ and constant $K_{P}$.
Assume that $R^*R\sim_s I$, and that 
$R$ satisfies a FAP$(\gamma,0)$ with respect to $AQ$, with  $\gamma > 0$ and constant $K_{R}$. (Note,
$\beta+\gamma > 1$). In addition, assume there exists  $k \leq n_c$ such that the decompositions of $R$ and $P$ in 
Lemma \ref{lem:PRdecomp} satisfy
\begin{enumerate}
\item $\delta_P:=\sigma_k^{\beta}K_{P}^{1/2}  < 1/\sqrt{2}$,  \qquad (denote $\widehat{K}_{P} := K_{P}/(1-\delta_{P}^2)$ )
\item $\delta_{R} := \sigma_k^{\gamma}K_{R}^{1/2} < 1/\sqrt{2}$,  \qquad (denote $\widehat{K}_{R} := K_{R}/(1-\delta_{R}^2)$ )
\item $ \delta_{PR}^2 := \sigma_k^{\beta+\gamma-1}(\widehat{K}_{P}\widehat{K}_{R})^{1/2} { < 1/2}$,
\item $\begin{aligned}
s_1 > \frac{\delta_{PR}^2}{(1-\delta_{PR}^2)}. \qquad\qquad\quad \tcb{(when ~~ k<n_c)}
\end{aligned}$
\item $\hat{\delta}_P^2 := \sigma_k^{2\beta-1}\widehat{K}_{P}  < 1$,

\end{enumerate}
Then, there exist constants, $0< c_0\leq c_1$, such that, $\forall$ $\mathbf{x}$
\begin{equation}\label{eq:equivIP_lem}
c_0 \leq \frac{\| (A_c^*A_c)^{1/2} \mathbf{x} \|^2}{\| P^*(QA) P\mathbf{x} \|^2} \leq c_1.
\end{equation}
The constants are specified below.
\end{lemma}

\begin{proof}

\tcb{The proof of the case in which $k=n_c$ is a simplification of the proof for the case $k<n_c$ and will be omitted. Assume $k<n_c$.}
Recall from Lemma \ref{lem:PRdecomp} and an appropriate choice of $k$ (by assumption), there are change of bases,
$\tilde{P} = PB_P$ and $\tilde{R} = RB_R$, such that
\begin{equation}\label{eqn:equivIP2}
\tilde{P} = V \tilde{\mcP}= \left[ V_1, V_2\right]\left[\begin{array}{cc} I_{k} & 0 \\ -\widehat{\mcN}_2\Sigma_1^\beta & \mcW_2 \end{array}\right],
	\quad
\tilde{R} =  U\tilde{\mcR}= \left[ U_1, U_2\right]\left[\begin{array}{cc} I_{k} & 0 \\ -\widehat{\mcM}_2\Sigma_1^\gamma & \mcZ_2 \end{array}\right],
\end{equation}
where 
\begin{enumerate}
\item $\mcW_2^*\Sigma_2\mcW_2 = \mcZ_2^*\Sigma_2\mcZ_2 = I$, 
\item $\mcW_2^*\widehat{\mcN}_2 = \mcZ_2^*\widehat{\mcM}_2 = \mathbf{0}$, 
\item $\| \widehat{\mcN}_2 \| \leq \widehat{K}_{P}^{1/2} $,  and $\| \widehat{\mcM}_2\| \leq \widehat{K}_{R}^{1/2}$,
\item $ \mcZ_2^*\Sigma_2\mcW_2 = S_2 =\diag{[ s_1, \ldots, s_{n_c-k}]} $ with $0 \leq s_1  \leq \ldots \leq s_{n_c-k} \leq 1$. 
\end{enumerate}
\vspace{1mm}
Here, $\mcP$ and $\mcR$ represent $\tilde{P}$ and $\tilde{R}$ transformed by the right and left singular 
vectors, respectively, to an $\ell^2$-space. By
assumption, there also exist constants such that
\begin{equation*}
\zeta_0 \leq \frac{\langle Px,Px\rangle}{\langle x,x\rangle} \leq \zeta_1,\hspace{4ex}
\xi_0 \leq \frac{\langle Rx,Rx\rangle}{\langle x,x\rangle} \leq \xi_1.
\end{equation*}
Using the proof of Lemma \ref{lem:PRdecomp},
\begin{align}\label{eqn:cob}
\frac{\zeta_0}{\max\{2,\frac{1}{\sigma_{k+1}}\}} &\leq \frac{\langle Px,Px\rangle}{\langle \tilde{P}x ,\tilde{P} x\rangle}  \leq  \zeta_1,\hspace{4ex}
\frac{\xi_0}{\max\{2,\frac{1}{\sigma_{k+1}}\}}  \leq \frac{\langle Rx,Rx\rangle}{\langle \tilde{R}x ,\tilde{R} x\rangle}  \leq  \xi_1.
\end{align}

Note that, if $\tilde{P} = PB_P$ and $\tilde{R}=RB_R$ satisfy  \eqref{eq:equivIP_lem}  with constants $\tilde{c}_0$ and $\tilde{c_1}$, 
then $R$ and $P$ satisfy  \eqref{eq:equivIP_lem} with constants 0
\tcbb{$c_0 = \tilde{c}_0/(\| B_P^{-1}\|\|B_R\|)^2 $ and $c_1 = (\| B_P\|\|B_R^{-1}\|)^2 \tilde{c}_1$.}
Thus, it is sufficient to establish bounds on \eqref{eq:equivIP_lem} with $R$ and $P$ replaced by $\tilde{R}$ and $\tilde{P}$. 
Further,
(\ref{eqn:equivIP2}) yields
\begin{equation*}
\frac{\langle \tilde{R}^*A\tilde{P} x, \tilde{R}^*A\tilde{P} x \rangle}{\langle \tilde{P}^*(QA) \tilde{P} x, \tilde{P}^*(QA) \tilde{P} x \rangle}
	= \frac{\langle \tmcR^*\Sigma\tmcP x, \tmcR^*\Sigma\tmcP x \rangle}{\langle \tmcP^*\Sigma\tmcP x, \mcP^*\Sigma\tmcP x \rangle}.
\end{equation*}
By transitivity of norm equivalence \cite{Faber:1990ed}, it is then sufficient to show
\begin{equation*}
\tmcR^*\Sigma\tmcP \sim_n \left[\begin{array}{cc} \Sigma_1 & 0 \\ 0 & I \end{array}\right] \sim_n \tmcP^*\Sigma\tmcP,
\end{equation*}
or, equivalently,
\begin{equation*}
\tmcR^*\Sigma\tmcP \left[\begin{array}{cc} \Sigma_1^{-1} & 0 \\ 0 & I \end{array}\right] \sim_n I \sim_n
	\tmcP^*\Sigma\tmcP\left[\begin{array}{cc} \Sigma_1^{-1} & 0 \\ 0 & I \end{array}\right].
\end{equation*}
Expanding,
\begin{align}
\tmcP^*\Sigma\tmcP\left[\begin{array}{cc} \Sigma_1^{-1} & 0 \\ 0 & I \end{array}\right]
	 &= \left[\begin{array}{cc} I + \Sigma_1^{\beta} \widehat{\mcN}_2^*\Sigma_2\widehat{\mcN}_2\Sigma_1^{\beta-1} &
	- \Sigma_1^{\beta} \widehat{\mcN}_2^*\Sigma_2\mcW_2\\ -\mcW_2^*\Sigma_2\widehat{\mcN}_2\Sigma_1^{\beta-1} & I \end{array}\right],
	 \label{eq:blockequiv2}\\
\tmcR^*\Sigma\tmcP\left[\begin{array}{cc} \Sigma_1^{-1} & 0 \\ 0 & I \end{array}\right] &=
	\left[\begin{array}{cc} I + \Sigma_1^{\gamma} \widehat{\mcM}_2^*\Sigma_2\widehat{\mcN}_2\Sigma_1^{\beta-1} 
	& -\Sigma_1^{\gamma} \widehat{\mcM}_2^*\Sigma_2\mcW_2\\ -\mcZ_2^*\Sigma_2\widehat{\mcN}_2\Sigma_1^{\beta-1}
	& S \end{array}\right].
	\label{eq:blockequiv1}
\end{align}

Next, we will invoke Lemma \ref{lem:bound_block} to bound the action of each of these operators from above and below. This will imply
norm equivalence to the identity and complete the proof. In each case, the diagonal blocks must be bounded from above and
below, and the off-diagonal blocks from above. In the case of the lower bound, there are additional requirements on the bounds
(see Lemma \ref{lem:bound_block}), which we verify are satisfied. For ease of notation, we denote bounds using notation of
Lemma \ref{lem:bound_block}:
\begin{align*}
\left[\begin{array}{cc} ~~A & -B \\ -C & ~~D \end{array}\right]  \hspace{3ex} \mapsto \hspace{3ex}
	\begin{array}{c r }
	a_0\|\mathbf{x}\| \leq \|A\mathbf{x}\| \leq a_1\|\mathbf{x}\|,& \hspace{4ex} \|B\mathbf{x}\| \leq b\|\mathbf{x}\|,\\
	d_0\|\mathbf{x}\| \leq \|D\mathbf{x}\| \leq d_1\|\mathbf{x}\|,& \hspace{4ex} \|C\mathbf{x}\| \leq c\|\mathbf{x}\|,
	\end{array}
\end{align*}
for $a_0,d_0 > 0$, $a_1,b,c,d_1\geq 0$, and $a_0d_0 > bc$. Most of these bounds have been shown previously and, in all cases,
follow naturally from the bases constructed in Section \ref{sec:2grid:basis} and Lemma \ref{lem:PRdecomp}. \\

\noindent
$\begin{aligned}
\textnormal{\underline{Equation \eqref{eq:blockequiv2}}:}\hspace{4ex}
	\|\mathbf{x}\| \leq &\|A\mathbf{x}\| \leq (1+\hat{\delta}_{P}^2)\|\mathbf{x}\|, \hspace{4ex} \| B \mathbf{x}\|\leq \sigma_k^{\beta}      \widehat{K}_{P}^{1/2}\|\mathbf{x}\|, \\
	\| \mathbf{x} \| \leq &\|D\mathbf{x}\| \leq \|\mathbf{x}\| , \hspace{12.25ex} \| C\mathbf{x} \| \leq \sigma_k^{\beta-1}
	\widehat{K}_{P}^{1/2}\|\mathbf{x}\|.
\end{aligned}$\vspace{1.5ex}

\noindent
Note that \tcb{$\|B\|\|C\| \leq \hat{\delta}_P^2$}. 
Here, $\beta \geq 1$, all terms are bounded independent of $\Sigma_1$, and the determinant bound $a_0d_0-bc = 1-\hat{\delta}_{P}^2 > 0$
is satisfied. Application of Lemma \ref{lem:bound_block} yields the result.\footnote{Slightly better bounds can be obtained for
$P^*QAP$ by directly proving spectral equivalence; however, the proof is longer and is not significant to the final result.}\\
\\
\noindent
$\begin{aligned}
\textnormal{\underline{Equation \eqref{eq:blockequiv1}}:}\hspace{4ex}
(1-\delta_{PR}^2) \|\mathbf{x}\| \leq  & \|A\mathbf{x}\| \leq  (1+ \delta_{PR}^2) \|\mathbf{x}\|, \hspace{4ex}
	\tcb{\|B\mathbf{x} \| \leq \sigma_k^{\gamma}\widehat{K}_{R}^{1/2}\|\mathbf{x}\|,} \\
	s_1\|\mathbf{x}\| \leq & \|D\mathbf{x} \| \leq \|\mathbf{x}\| , \hspace{14ex}
	\tcb{\|C\mathbf{x}\| \leq \sigma_k^{\beta-1}\widehat{K}_{P}^{1/2}\|\mathbf{x}\|.}
\end{aligned}$\\

\noindent
\tcb{Note that $\|B\| \|C\| \leq \delta_{PR} < 1/2$ by Hypothesis $5$.} 
Lemma \ref{lem:bound_block} applies here if each term is bounded independent of $\Sigma_1$ and
$s_1 >  \frac{\delta_{PR}^2}{ (1-\delta_{PR}^2)}$, which is ensured by Hypothesis $4$ and the assumption that
$\beta+\gamma > 1$. 

Constants $\tilde{c}_0$, $\tilde{c_1}$ can be found by applying Lemma \ref{lem:bound_block}. Finally, 
(\ref{eqn:cob}) may be used to find $c_0$ and $c_1$.
\end{proof}

\begin{corollary}
If all assumptions in Lemma \ref{lem:equivIP} are independent of grid level, then $R^*AP\sim_nP^*QAP$, with
constants independent of grid level.
\end{corollary}

With this line of proof, it is clear why $P$ must have at least a SAP/SSAP for inner-product equivalence, that is, $\beta \geq 1$. If not,
then $\|A\|$ and $\|C\|$ in \eqref{eq:blockequiv2} and \eqref{eq:blockequiv1} are not bounded independent of $\Sigma_1$. 
Also note that $R$ plays a minor role. Although stronger
approximation properties for $R$ (larger $\gamma$) improve the equivalence constants, $\gamma$ is only required by the
proof to satisfy $\beta+\gamma > 1$ and $\gamma > 0$. Of course, everything is made easier by choosing $R$ to be close to
$Q^*P$ and/or to share the same FAP power as $P$.

\begin{remark}
The same relation between $\sigma_k$ and the constraint on $s_1$ discussed in Section \ref{sec:2grid:stability} for stability
applies here as well. The definitions of $B$ in \eqref{eq:blockequiv2} and \eqref{eq:blockequiv1} are slightly different, but
satisfy the same properties. As $k$ is chosen smaller, 
$\delta_{PR}^2$ and $\hat{\delta}_{P}^2$ get smaller, which reduces the constraint on $s_1$. However, smaller $k$ also leads to
$W_2$ and $Z_2$ of larger dimensions, which likely makes $s_1$ smaller.
\end{remark}

\subsection{Multilevel convergence}\label{sec:multigrid:conv}

So far we have considered the relation between the orthogonal coarse-grid operator and coarse-grid operator used in practice.
To prove multilevel convergence, we will decompose error over the subspaces $\mathcal{R}(\Pi)$ and $\mathcal{R}(I-\Pi)$.
For an orthogonal projection, say $\widehat{\Pi}$ with respect to norm $\|\cdot\|$, $\|\mathbf{e}\|^2 =
\|(I - \widehat{\Pi})\mathbf{e}\|^2 + \|\widehat{\Pi}\mathbf{e}\|^2$. Because $\Pi$ as
used here is a non-orthogonal projection, this equality does not hold. However, bounds on the decomposition are closely related to
stability as proved in Section \ref{sec:2grid:stability}, and the angle between the subspaces $\mathcal{R}(\Pi)$ and
$\mathcal{R}(I-\Pi)$. 

From a given level in the AMG hierarchy, denote the coarse-grid matrix $A_c$, and define $Q_cA_c:=(A_c^*A_c)^{\frac{1}{2}}$,
where $Q_cA_c$ defines the norm we will consider on the coarse grid. Then, consider the difference between the exact projection,
$\Pi = PA_c^{-1}R^*A$, and the inexact projection, $\widetilde{\Pi} = PB_c^{-1}R^*A$, where $B_c^{-1}$ denotes the AMG cycle applied
to the coarse-grid problem. This corresponds to the recursive application of a multilevel AMG cycle. Assume $B_c^{-1}$ is convergent,
with bound
\begin{align*}
\|I - B_c^{-1}A_c\|_{Q_cA_c}^2 = \|(A_c^{-1} - B_c^{-1})A_c\|_{Q_cA_c}^2 < \rho_c,
\end{align*}
and let $G^\nu$ denote the error-propagation operator corresponding to
$\nu$ iterations of relaxation. Then, from \eqref{eq:equiv9} and Lemma \ref{lem:equivIP},
\begin{align*}
\|(\Pi - \widetilde{\Pi})G^\nu\mathbf{e}^{(i)}\|_{QA}^2 & = \left\|P(A_c^{-1}-B_c^{-1})A_c(A_c^{-1}R^*AG^\nu)\mathbf{e}^{(i)}\right\|_{QA}^2 \\
& \leq \frac{1}{c_0} \Big\|( A_c^{-1}-B_c^{-1})A_c(A_c^{-1}R^*AG^\nu)\mathbf{e}^{(i)}\Big\|_{Q_cA_c}^2 \\
& \leq \frac{\rho_c}{c_0} \left\|A_c^{-1}R^*AG^\nu\mathbf{e}^{(i)}\right\|_{Q_cA_c}^2 \\
& \leq \frac{c_1\rho_c}{c_0} \left\| PA_c^{-1}R^*AG^\nu\mathbf{e}^{(i)}\right\|_{QA}^2 \\
& = \frac{c_1\rho_c}{c_0}\|\Pi G^\nu\mathbf{e}^{(i)}\|_{QA}^2.
\end{align*}
Error is propagated via $\mathbf{e}^{(i+1)} = (I - \widetilde{\Pi}) G^\nu \mathbf{e}^{(i)}$, which can be expanded in norm as
\begin{align}
\begin{split}
\|\mathbf{e}^{(i+1)}\|_{QA}^2 & \leq \|(I - \Pi)G^\nu\mathbf{e}^{(i)}\|_{QA}^2 + 2\left\langle (I - \Pi)G^\nu\mathbf{e}^{(i)}, (\Pi - \widetilde{\Pi})G^\nu\mathbf{e}^{(i)}\right\rangle_{QA}
	\\&\hspace{6ex} + \|(\Pi - \widetilde{\Pi})G^\nu\mathbf{e}^{(i)}\|_{QA}^2.
\end{split}\label{eq:error_exp}
\end{align}
In order to bound the middle inner product, we introduce the following result connecting the angle between subspaces
of a Hilbert space, the norm of an oblique projection, and a strengthened Cauchy-Schwarz inequality. 
\begin{lemma}[Strengthened Cauchy Schwarz]\label{lem:sCS} 
Define the minimal canonical angle between $\mathcal{R}(\Pi)$ and $\mathcal{R}(I-\Pi)$ in the $QA$ inner product by 
\begin{align*}
\cos\left(\theta_{min}^{(\Pi)}\right) := \sup_{\substack{\mathbf{x}\in\mathcal{R}(\Pi), \|\mathbf{x}\|_{QA}=1, \\\mathbf{y}\in\mathcal{R}(I-\Pi),
	\|\mathbf{y}\|_{QA}=1}} |\langle \mathbf{x},\mathbf{y}\rangle_{QA} |.
\end{align*}
Then, 
$\|\Pi\|_{QA}  = \|I - \Pi\|_{QA} = \frac{1}{\sin\left(\theta_{min}^{(\Pi)}\right)}$, 
and, for all $\mathbf{x}\in\mathcal{R}(\Pi)$ and $\mathbf{y}\in\mathcal{R}(I - \Pi)$, 
$\Big|\langle \mathbf{x}, \mathbf{y}\rangle_{QA} \Big| \leq \cos\left(\theta_{min}^{(\Pi)}\right) \|\mathbf{x}\|_{QA}\|\mathbf{y}\|_{QA}.$
\end{lemma}
\begin{proof}
See \cite{Deutsch:1995tc,Szyld:2006bg}.
\end{proof}

Applying Lemma \ref{lem:sCS} and an $\epsilon$-inequality with $\epsilon=1$ to \eqref{eq:error_exp} yields
\begin{align*}
\|\mathbf{e}^{(i+1)}\|_{QA}^2 & \leq \|(I - \Pi)G^\nu\mathbf{e}^{(i)}\|_{QA}^2 + \frac{c_1\rho_1}{c_0} \|\Pi G^\nu\mathbf{e}^{(i)}\|_{QA}^2 \\ &\hspace{6ex} +
	2\cos\left(\theta_{min}^{(\Pi)}\right)\|(I - \Pi)G^\nu\mathbf{e}^{(i)}\|_{QA}\sqrt{\frac{c_1\rho_c}{c_0}}\|\Pi G^\nu\mathbf{e}^{(i)}\|_{QA}  \\
& \leq \left( 1 + \cos^2\left(\theta_{min}^{(\Pi)}\right)\right)\|(I - \Pi)G^\nu\mathbf{e}^{(i)}\|_{QA}^2 + 
	2\frac{c_1\rho_c}{c_0}\|\Pi G^\nu\mathbf{e}^{(i)}\|_{QA}^2,
\end{align*}
for angle $\theta_{min}^{(\Pi)}$ between $\mathcal{R}({\Pi})$ and $\mathcal{R}(I - {\Pi})$.
Here, the first term corresponds to error that is not in the range of interpolation and must be attenuated by relaxation, while the
second term is the error that is in the range of interpolation, but has not been eliminated by the inexact coarse-grid
correction. Then, \tcb{using the last statement of Lemma \ref{lem:sCS}},
\begin{align}\label{eq:cpi}
C_\Pi\left(1 + \cos^2\left(\theta_{min}^{(\Pi)}\right)\right) = C_\Pi\left(2 - \sin^2\left(\theta_{min}^{(\Pi)}\right)\right) \leq 2C_\Pi - 1. 
\end{align}

{\color{black}
Let $G^\nu$ correspond to $\nu$ iteration of Richardson relaxation on the normal equations. By Corollary \ref{cor:2grid} and \eqref{eq:cpi},
\begin{align}\label{eq:cf-ML}
\|\mathbf{e}^{(i+1)}\|_{QA}^2 & \leq \rho_{\nu,\beta} \|\mathbf{e}^{(i)}\|_{QA}^2 + 
	\frac{2 c_1 C_\Pi \rho_c}{c_0}\|\mathbf{e}^{(i)}\|_{QA}^2,
\end{align}
where
\begin{align}\label{eq:rho_nu}
\rho_{\nu,\beta} &= \left(\frac{(2\beta-1)}{4\nu+(2\beta-1)}\right)^{(2\beta-1)/2} K_{P,\beta,1} (2 C_\Pi -1).
\end{align}
Assume the same constants hold on all levels and 
let $L$ designate the coarsest level in the hierarchy, where the coarse grid is solved exactly. Thus, $\rho_c= \rho_{_L} = 0$. On the next level,
using $G^\nu$, the convergence factor satisfies
\begin{align*}
\rho_{_{L-1}} &\leq  C_\mu \rho_{_L} + \rho_{\nu,\beta} = \rho_{\nu,\beta},
\end{align*}
where $C_\mu = 2 (c_1/c_0)C_\Pi.$
Thus, the AMG preconditioner corresponding to the inexact solve of level $L-1$ has convergence factor
$\|I - B_{L-1}^{-1}A_{L-1}\|_{Q_{L-1}A_{L-1}}^2 \leq \rho_{\nu,\beta}$.

Moving up the hierarchy, on level $L-2$, let  $\mu$ be the number of AMG cycles
applied as an inexact solve. Then, $\rho_c = \rho_{_{L-1}}^\mu$ in \eqref{eq:cf-ML} and
\begin{align}
\rho_{_{L-2}} &\leq C_\mu \rho_{_{L-1}}^\mu + \rho_{\nu,\beta}.
\end{align}
Thus, $\|I - B_{L-2}^{-1}A_{L-2}\|_{Q_{L-2}A_{L-2}}^2 \leq \rho_{_{L-2}}$. Given, $\beta$ and $\mu$ is there a value of $\nu$
for which this recursion is bounded? Since $C_\mu > 1.0$, $\mu = 1$ will not work. Assume $\mu = 2$, corresponding to a
W-cycle. If $\nu$ is chosen such that  $\rho_{\nu,\beta} \leq 1/(4C_\mu)$, then
\begin{align}\label{eq:cf-mu=2}
\rho_\ell &\leq 1/(2C_\mu)  < 1,
\end{align}
for all $\ell \leq L$. Appealing to \eqref{eq:rho_nu}, this is satisfied if
\begin{align}\label{eq:nu-ML}
\nu &\geq \left((2\beta-1)/4\right)\left( 4(c_1/c_0)K_{P} C_\Pi(2C_\Pi-1)\right)^{2/(2\beta-1)}.
\end{align}
For $\mu > 2$, a similar argument will yield a less stringent condition on $\nu$, which we omit.

From Corollary \ref{cor:2grid}, the constant in \eqref{eq:rho_nu} is $K_{P,\beta,1}$, which can be much smaller than $K_{P,\beta,0}$, 
as will be shown  numerically in Section \ref{sec:numerics}.
If $P$ satisfies a SAP ($(\beta,\eta) =  (1,1)$), then the number of relaxations grows like $O(K_{P,1,1}^2 C_\Pi^4)$. If
$P$ satisfies a FAP$(3/2,1)$, then the number of relaxations grows like $O(K_{P,3/2,1} C_\Pi^2)$. This emphasizes the 
goal of choosing $R$ and $P$ to increase $\beta$ and reduce $C_\mu$ and $K_{P,\beta,\eta}$. 

The discussion above is summarized in the following theorem, where $W$-cycle 
convergence is established. Proof for $\mu$-cycle would follow similarly.

\begin{theorem}[$W$-cycle Convergence]\label{th:W-cycle}
Consider an AMG hierarchy with $L$ levels, and assume the conditions for Lemma \ref{lem:equivIP} hold on each level.
Let the constants, including $c_1, K_{P}$, and $C_\Pi$ denote the maximum corresponding values over
all levels in the hierarchy, and $c_0$ the minimum value over all levels. Set $C_\mu =  2(c_1/c_0) K_{P} C_\Pi$
and choose
\begin{align}\label{eq:nu-W}
\nu &\geq \left((2\beta-1)/4\right)\left( 4(c_1/c_0)K_{P} C_\Pi(2C_\Pi-1)\right)^{2/(2\beta-1)}.
\end{align}
Then, $W$-cycle convergence factor is bounded by
\begin{align}\label{eq:cf-W}
\rho &\leq 1/(2C_\mu)  < 1.
\end{align}
\end{theorem}

\begin{proof}
The proof follows from the discussion above.
\end{proof}

Theorem \ref{th:W-cycle} proves the existence of a convergent, $W$-cycle, with convergence independent of the problem
size and number of levels in the hierarchy. A $W$-cycle is scalable as long as the coarsening ratio, defined to be the ratio
of the DOFs in the coarse grid divided by the DOFs on the fine grid, is less than $1/2$. This is important for application to 
hyperbolic problems, for which a coarsening ratio of approximately $1/2$ is expected \cite{air1}. The same approach could be used to prove 
convergence of  $\mu$-cycle with $\mu> 2$, which could be accomplished with a smaller number
of relaxations, $\nu$, but would only be scalable with more aggressive coarsening. 

}

\section{Numerical results} \label{sec:numerics}
{\color{black}

This section evaluates the norm of projections and approximation property constants for two highly nonsymmetric
discretizations of the two-dimensional linear steady state advection problem,
\begin{align}\label{eq:transport}
\begin{split}
\mathbf{b}(x,y) \cdot\nabla u & = q(x,y) \hspace{3ex}\Omega, \\
u & = g(x,y) \hspace{3ex}\Gamma_{\textnormal{in}},
\end{split}
\end{align}
for domain, $\Omega \in\mathbb{R}^{2}$, and inflow boundary $\Gamma_{\textnormal{in}}$. A scalar PDE is chosen
to avoid complications that arise from satisfying approximation properties for systems of PDEs, and a purely advective
problem is chosen so that the resulting discretizations are highly nonsymmetric, independent of mesh spacing, $h$
(whereas advection-diffusion, for example, becomes increasingly symmetric as $h\to 0$). 

The domain $\Omega = [0,1]\times[0,1]$ is discretized using an unstructured triangular mesh, and the velocity field
given by a constant direction $\mathbf{b}(x,y) = ( \cos(\theta),\sin(\theta))$, where $\theta = 3\pi/16$. Inflow
boundary conditions are imposed on the south and west boundaries with $g =1$. Equation \eqref{eq:transport} is discretized
using upwind discontinuous Galerkin (DG) \cite{Brezzi:2004hf} and streamline upwind Petrov-Galerkin (SUPG) \cite{Brooks:1982bl}
discretizations. The resulting matrices are then scaled by the (block) diagonal to approximately account for relaxation
before considering the approximation properties. 
Similar results have been obtained for various curved velocity fields as well as including a reaction term, but
here we focus on the simpler case of constant advection. For numerical tests, relatively small spatial domains are
considered, $20\times 20$ for DG and $50\times 50$ for SUPG, each leading to about 3000 total DOFs, which
is necessary to directly evaluate the projections and approximation properties. 

Two methods are considered for computing transfer operators, a classical AMG interpolation operator \cite{Ruge:1987vg},
which is widely used and known to be effective for many scalar elliptic problems, and a restriction operator based on
a local approximate ideal restriction, $\ell$AIR \cite{air2}. 
Recently, the $\ell$AIR restriction was shown to be effective on highly
nonsymmetric matrices when coupled with relatively simple interpolation operators. 
In particular, the linear advection and transport equations were examined in \cite{air1,air2}. In \cite{air1}, a
reduction-based framework for convergence of NS-AMG is developed to explain the strong convergence obtained
using $\ell$AIR on hyperbolic-type problems. However, here we see that, in fact, $\ell$AIR also has good approximation
properties. Results here also consider classical AMG interpolation used as a restriction operator, $R = P$,
as occurs when using a Galerkin coarse grid, and an equivalent $\ell$AIR-like algorithm {on $A^*$}  to approximate 
the ideal interpolation operator, referred to as a local
approximate ideal prolongation ($\ell$AIP). 
\tcbb{
Figure \ref{fig:waps} shows the WAP (FAP$(1/2,0)$), SAP (FAP$(1,1)$)
and SSAP (FAP$(1,0)$) approximation constants for each individual (left/right) singular vector of $A$. Horizontal 
lines indicate the approximation constant that holds for all vectors.
}

\tcb{
Of interest is the behavior of the constant associated with individual singular vectors as the corresponding singular value
becomes small. If the constant remains bounded (flat or decreasing with decreasing singular value), this suggests the
particular approximation property holds independent of
problem size. If the constant spikes, it is an indication that the  property likely does not hold independent of problem size. If they tend
toward zero, it suggests a higher approximation property might also hold.
Recall that Lemma \ref{lem:approxprop} proves that a SAP implies a SSAP with constant squared. This behavior is demonstrated 
by the much larger values for the SSAP than for the SAP.
}

\begin{figure}[!ht]
  \centering
    \begin{subfigure}[b]{\textwidth}
      \begin{subfigure}[b]{0.475\textwidth}
        \includegraphics[width=\textwidth]{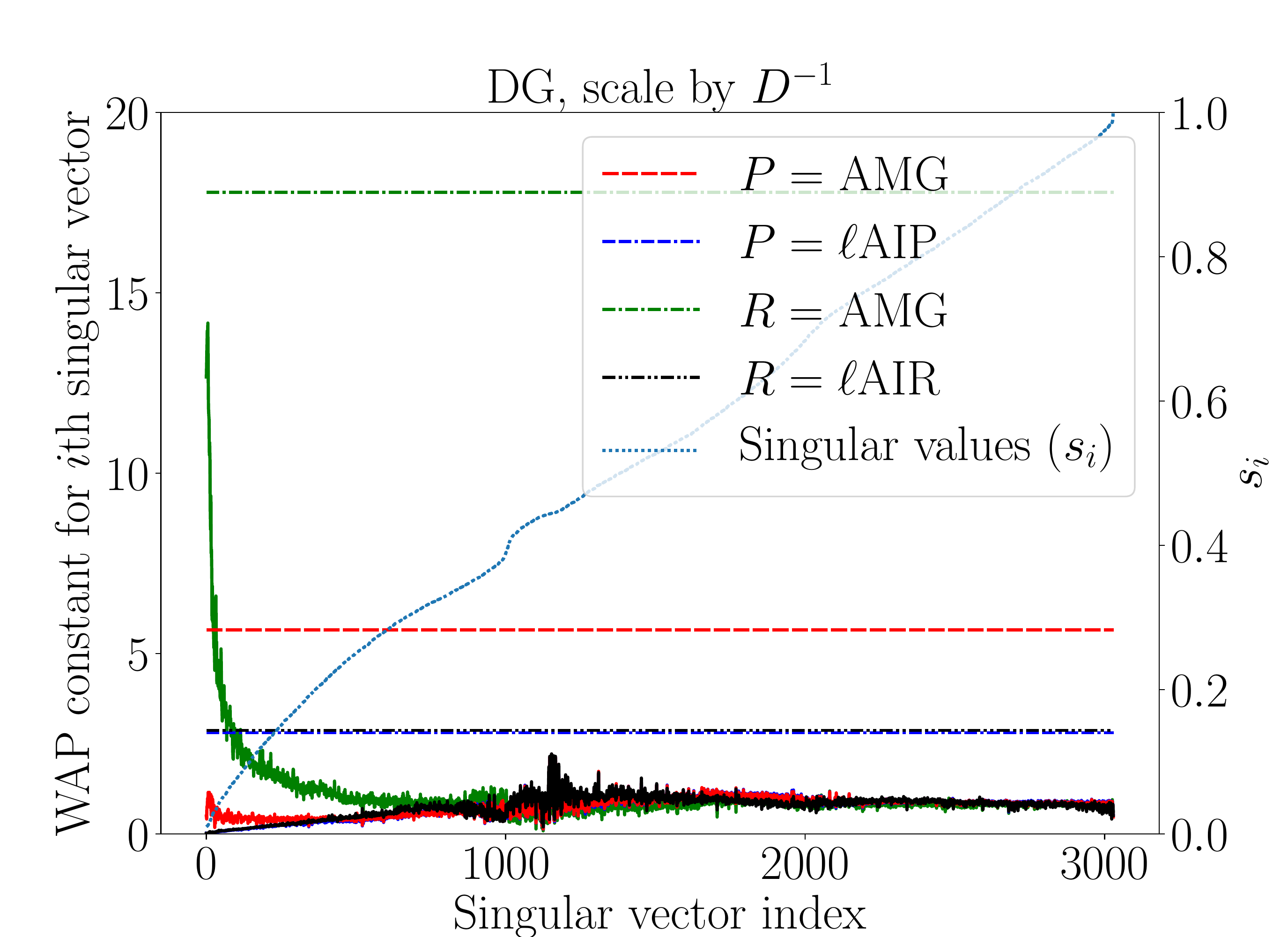}
      \end{subfigure}
      \hfill
       \begin{subfigure}[b]{0.475\textwidth}
        \includegraphics[width=\textwidth]{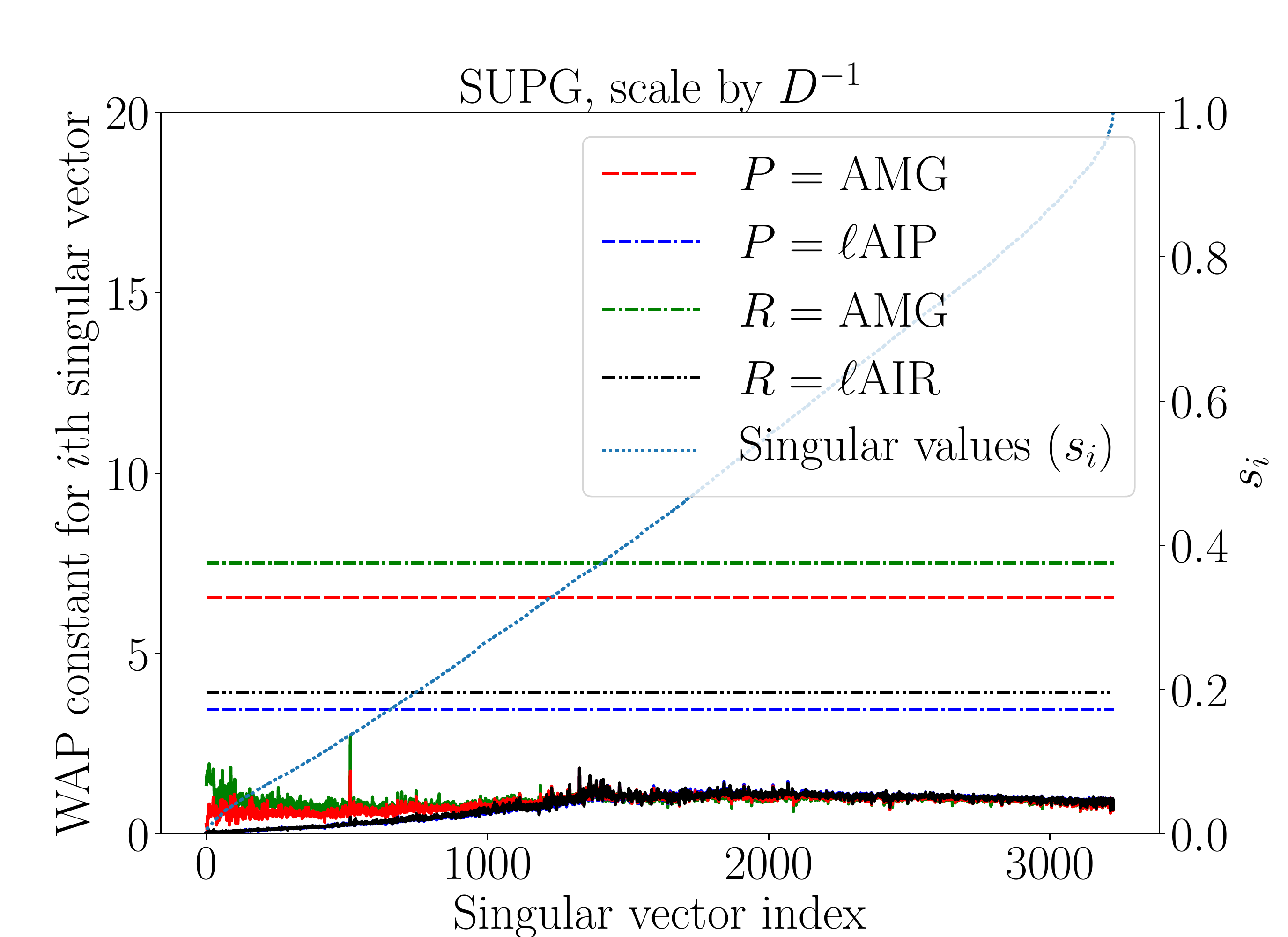}
      \end{subfigure}
  \label{fig:inset}
  \end{subfigure}
  \\
    \begin{subfigure}[b]{\textwidth}
      \begin{subfigure}[b]{0.475\textwidth}
        \includegraphics[width=\textwidth]{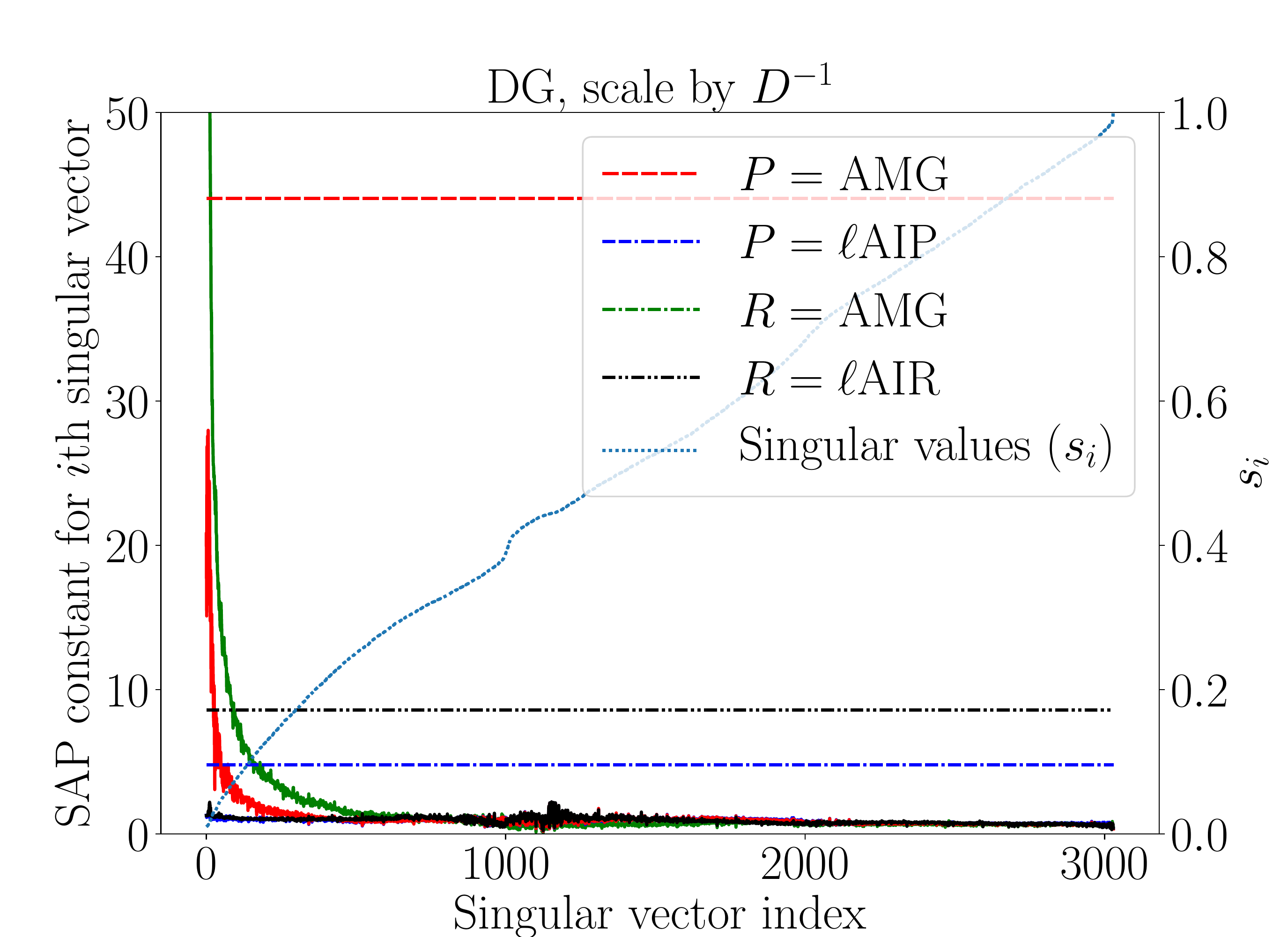}
      \end{subfigure}
      \hfill
       \begin{subfigure}[b]{0.475\textwidth}
        \includegraphics[width=\textwidth]{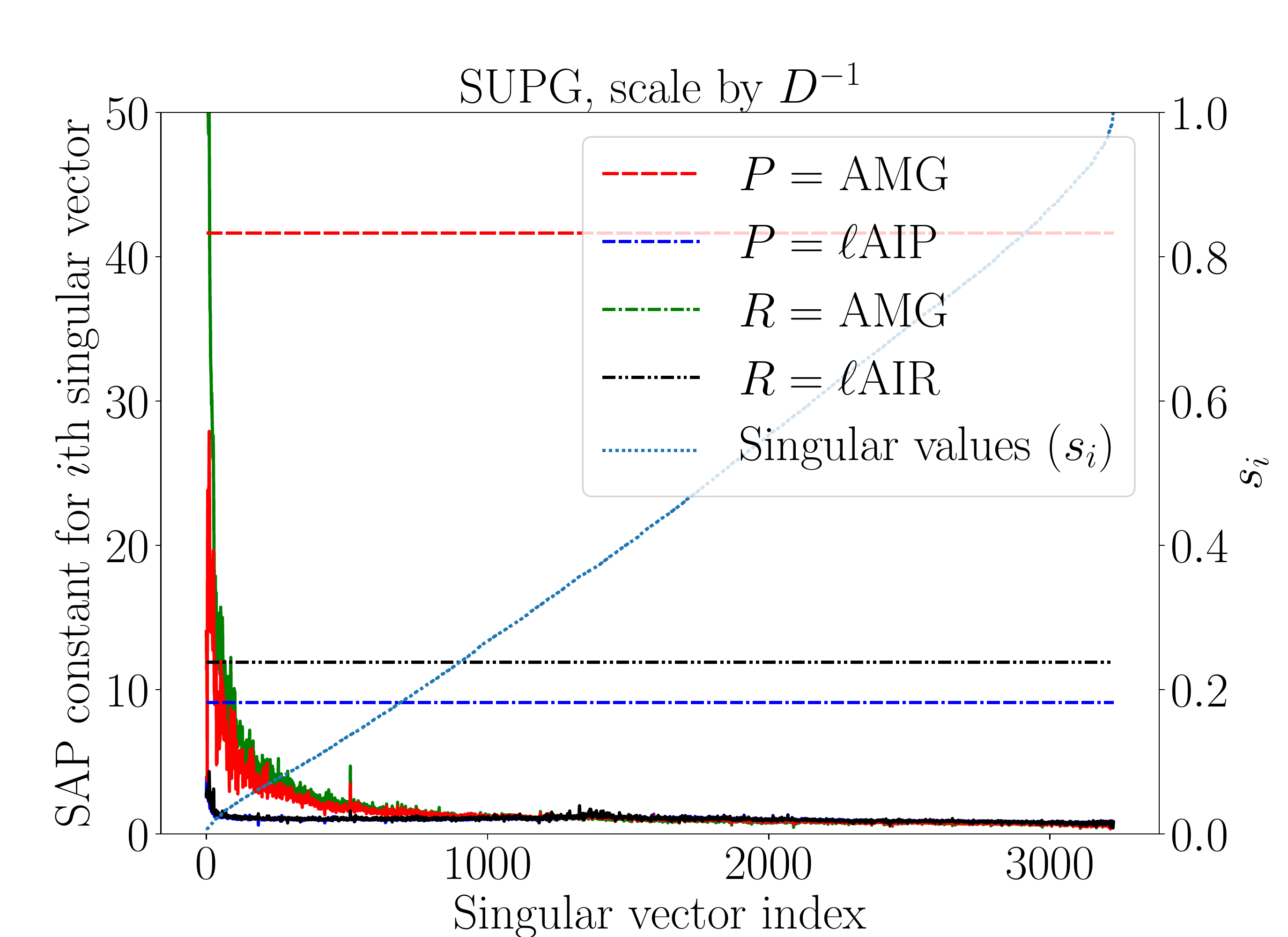}
      \end{subfigure}
  \end{subfigure}
  \\
      \begin{subfigure}[b]{\textwidth}
      \begin{subfigure}[b]{0.475\textwidth}
        \includegraphics[width=\textwidth]{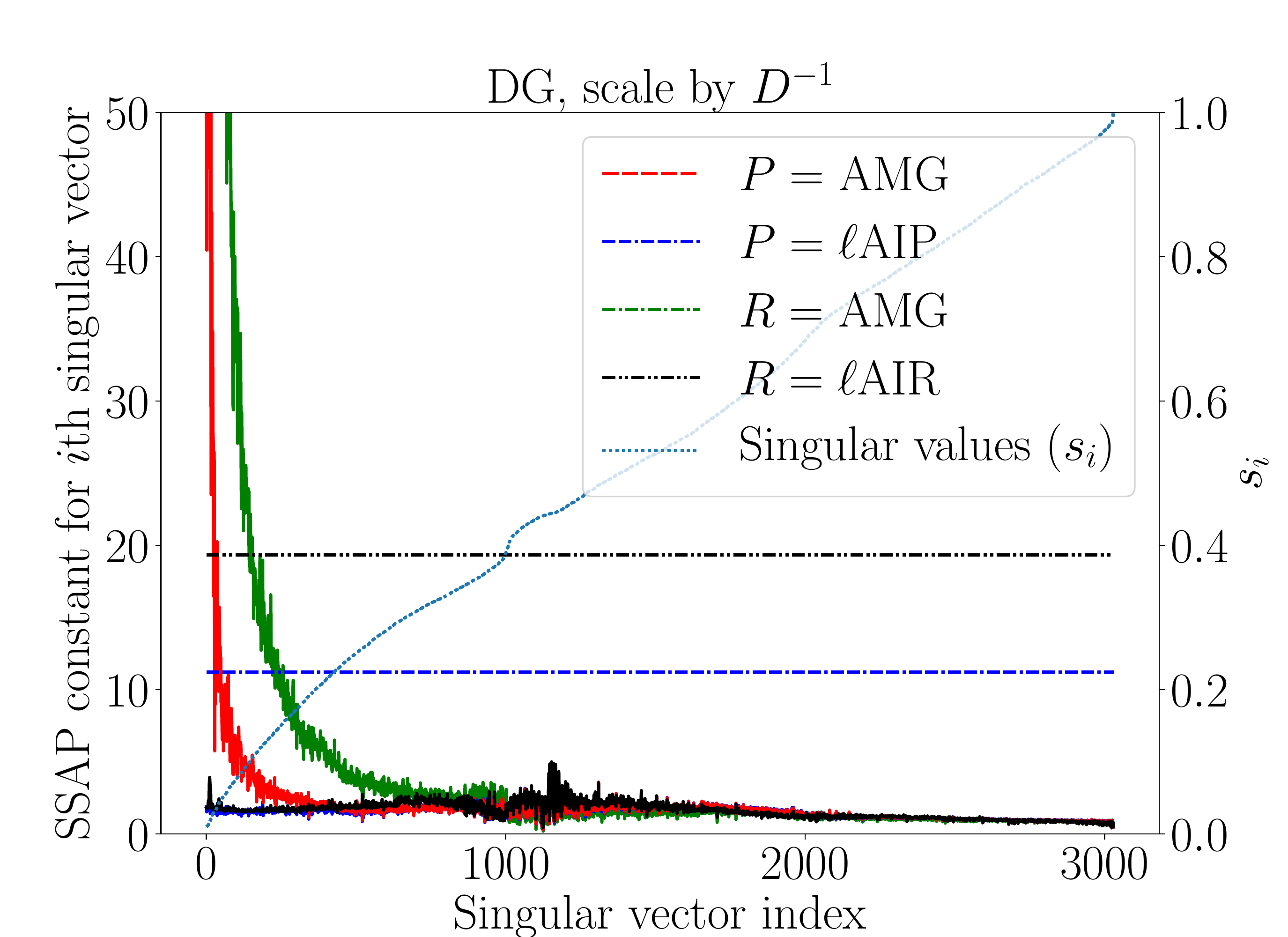}
      \end{subfigure}
      \hfill
       \begin{subfigure}[b]{0.475\textwidth}
        \includegraphics[width=\textwidth]{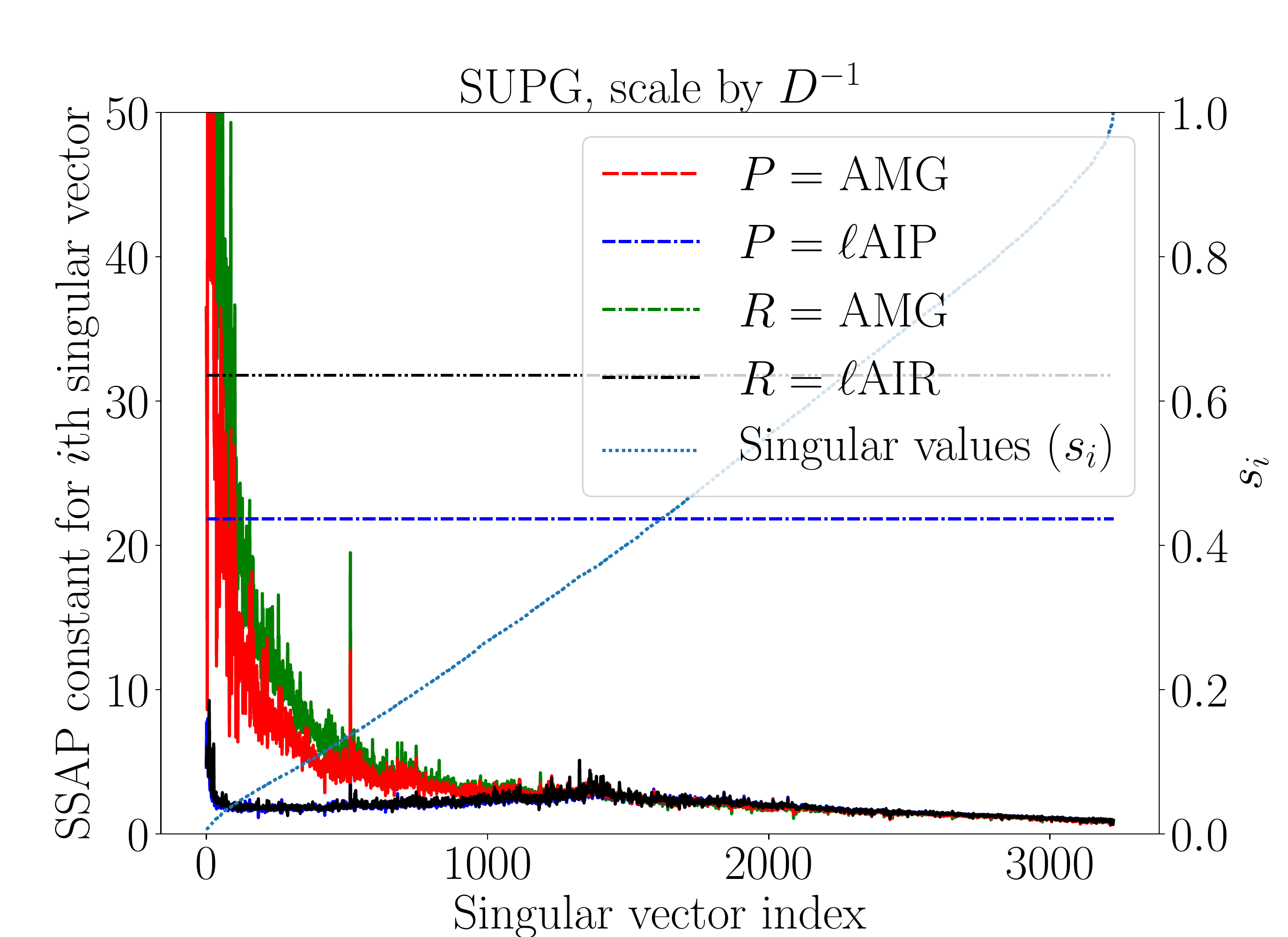}
      \end{subfigure}
  \end{subfigure}
    \caption{\tcb{WAP, SAP, and SSAP constants for classical AMG and $\ell$AIP as interpolation operators,
    and classical AMG and $\ell$AIR as restriction operators. Singular values of $A$ are
    shown in dotted blue on the right axis, and solid lines show the approximation constant for each respective
    singular vector (right singular vector for interpolation, left for restriction). Horizontal dot/dashed lines of the corresponding
    color show the approximation property constant that holds for all vectors. {The axes were kept small to better focus on
    the best values.  The values that fall off the figure are omitted, but listed here.} The SAP for $R = P=$ AMG
    is 89 for DG and 68 for SUPG. Similarly, for DG and SUPG, respectively, the SSAPs for $P=$ AMG are 204 and 157,
    and SSAPs for $R = P = $ AMG are 1390 and 257.}
   }
  \label{fig:waps}
  \end{figure}

There are a number of interesting things to note from Figure \ref{fig:waps}:
\begin{itemize}
\item Classical AMG, {indicated in red and} known to be effective on scalar elliptic PDEs, is not a good interpolation operator
for these problems. Although it may have a WAP for DG and SUPG,
it clearly does not satisfy the stronger approximation properties, indicated by the spike in the constant for small singular values.

\item Using classical AMG interpolation as a restriction operator, $R = P$ (green), acts as an even worse restriction 
operator, exposing one of the difficulties of Galerkin-based AMG on highly nonsymmetric problems. In the single instance
where the corresponding WAP constant is only moderate in size (top right), the constant is still likely to increase
as $h\to 0$ because the least accurate approximation of singular vectors is on those with small singular values.

\item $\ell$AIR {(black)}, in addition to having good reduction-type properties as shown in \cite{air1}, also has good
approximation properties. Indeed, for DG,  $\ell$AIR appears to have a WAP, SAP, and thus, a SSAP, with
fairly small constants that are independent of problem size. For SUPG, the SAP and SSAP constants rise slightly for very small 
singular values. This leaves the exact approximation properties of $\ell$AIR on SUPG in question.

\item Interestingly, the interpolation method referred to as $\ell$AIP {(blue)} also has very good approximation properties, better
than all other grid-transfer operators tested here.
The algorithm  described in \cite{air1} in which $\ell$AIR is paired with a simple interpolation, was shown to converge well for highly nonsymmetric problems.
However, theory suggests that  a good restriction operator alone will not be sufficient for
scalable convergence in that context. Results here indicate that commonly used interpolation methods may not be 
as accurate as $\ell$AIP.  This suggests that $\ell$AIR paired with $\ell$AIP may provide a robust and scalable method
for this class of nonsymmetric systems.
This is a topic of current research
\end{itemize}

In addition to approximation properties, stability of coarse-grid correction is important for scalable
convergence. Figure \ref{fig:projs} plots the $\ell^2$- and $QA$-norms for various coarse-grid corrections,
including the Galerkin case ($R = P$), the Petrov-Galerkin case ($R\neq P$), and the orthogonal projection
in each respective norm. The $\ell^2$-norm is considered because $\ell$AIR approximates the ideal restriction
operator, which is ideal in a certain sense in the $\ell^2$-norm \cite{air1,air2}. Similar to Figure \ref{fig:waps},
the norm is plotted as a function of every right singular vector, with a horizontal line of the same color 
giving the full operator norm. 

\begin{figure}[!h]
  \centering
    \begin{subfigure}[b]{\textwidth}
      \begin{subfigure}[b]{0.475\textwidth}
        \includegraphics[width=\textwidth]{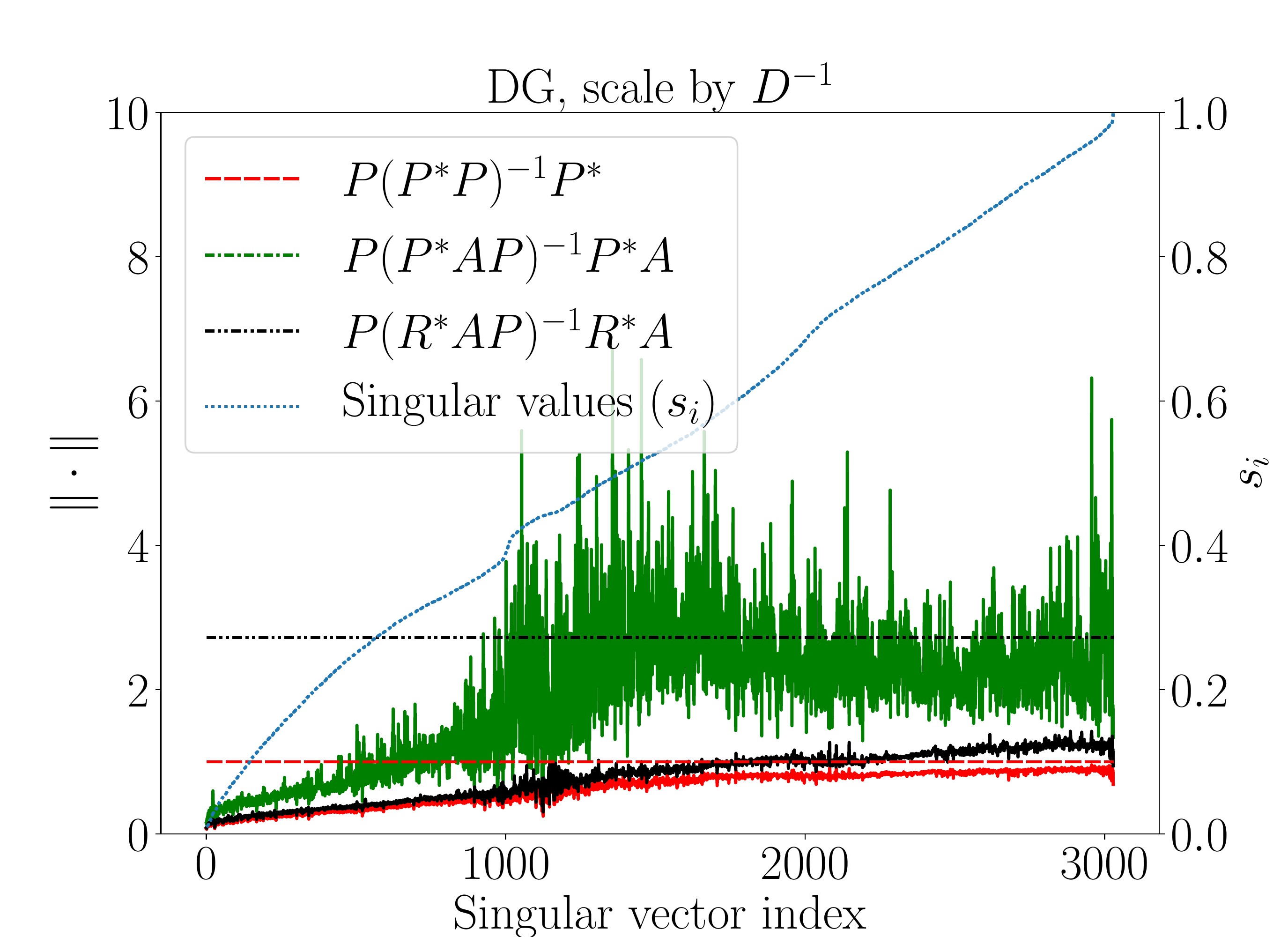}
      \end{subfigure}
      \hfill
       \begin{subfigure}[b]{0.475\textwidth}
        \includegraphics[width=\textwidth]{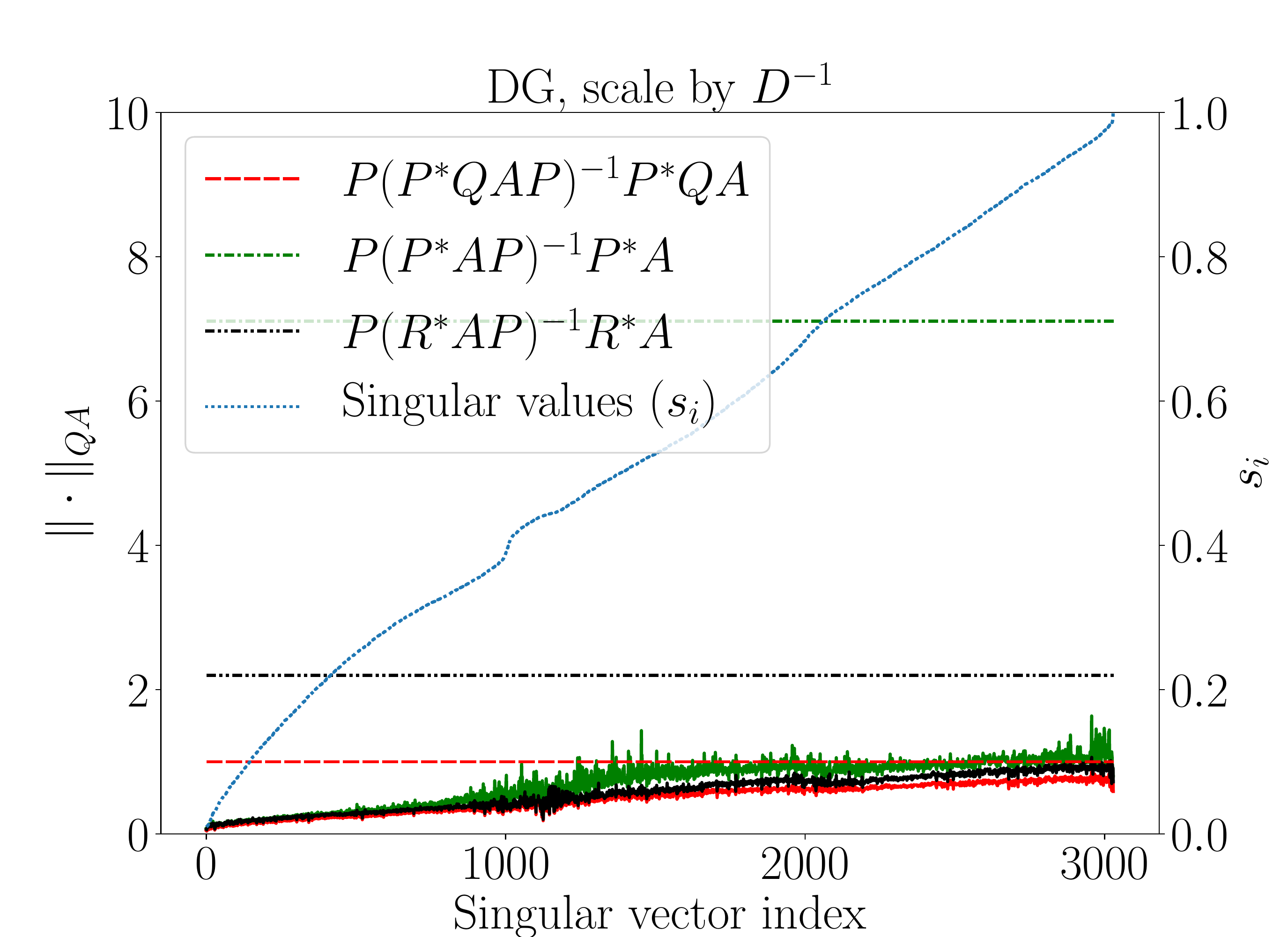}
      \end{subfigure}
  \end{subfigure}
  \\
    \begin{subfigure}[b]{\textwidth}
      \begin{subfigure}[b]{0.475\textwidth}
        \includegraphics[width=\textwidth]{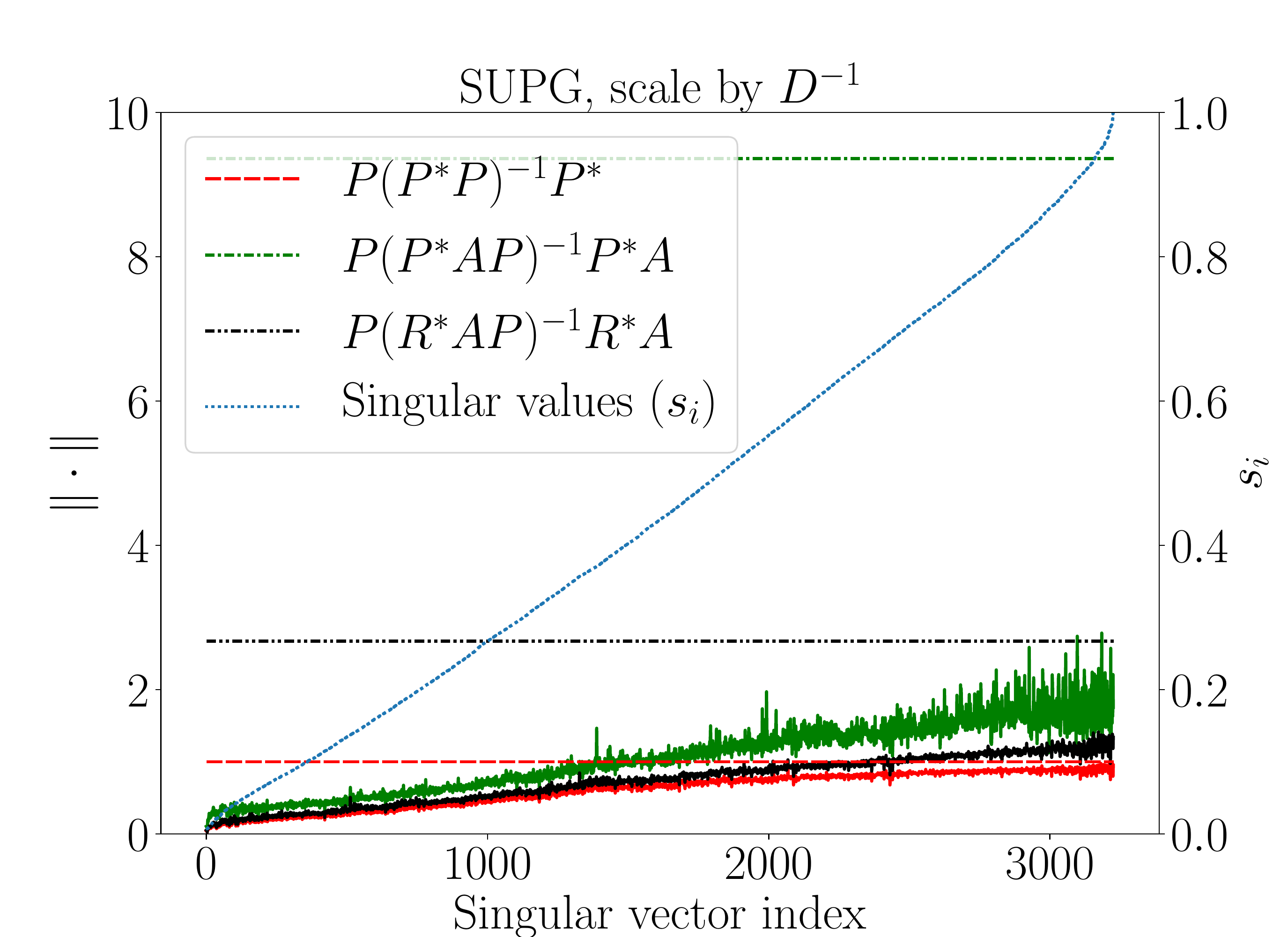}
      \end{subfigure}
      \hfill
       \begin{subfigure}[b]{0.475\textwidth}
        \includegraphics[width=\textwidth]{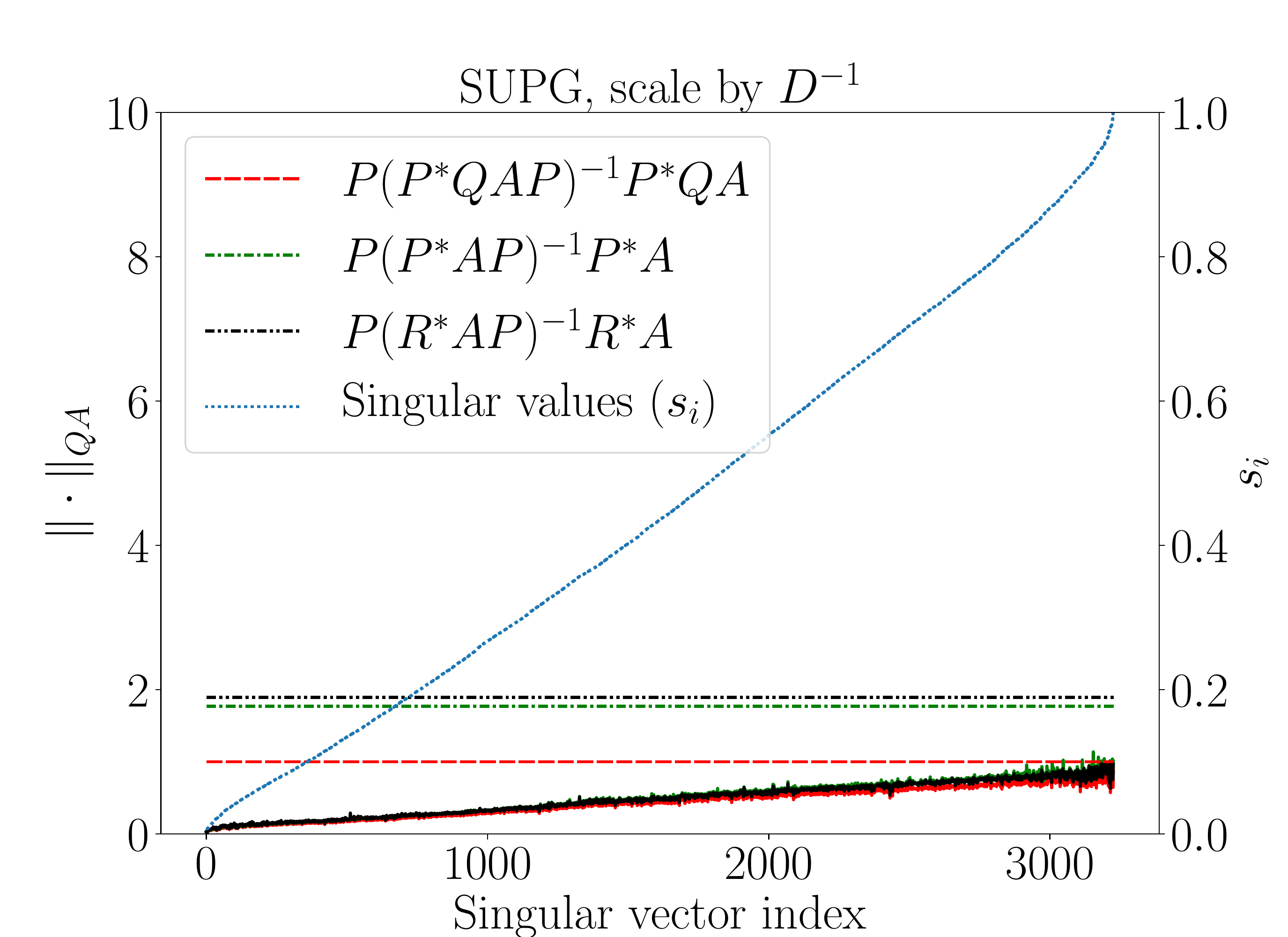}
      \end{subfigure}
  \end{subfigure}
    \caption{\tcb{$\ell^2$-norm (left) and $QA$-norm (right) of various projection operators, with $P$ given by classical
    AMG interpolation and $R$ given by $\ell$AIR. Singular values of $A$ are
    shown in dotted blue on the right axis, and solid lines show the norm of each projection applied to each respective
    right singular vector. Horizontal dot/dashed lines of the corresponding color show the actual norm of the projection. 
    The orthogonal (in the appropriate norm) projection onto the range of $P$ is shown in red 
    (and takes the value $1.0$), a Galerkin ($R = P$)
    projection is shown in green, and a Petrov-Galerkin ($R\neq P$) is shown in black. In the upper left plot, 
    $\|P(P^*AP)^{-1}P^*A\| \approx 30$ and is not shown.}
     }
  \label{fig:projs}
  \end{figure}

In all cases, the Petrov-Galerkin coarse-grid correction based on classical AMG interpolation and $\ell$AIR restriction
is nicely bounded in norm between 2--3. This further supports the Petrov-Galerkin approach over a Galerkin
coarse-grid correction, where, in three of the four cases here, the Galerkin projection is significantly larger in
norm. It is also important to note that the singular vectors which are most amplified by coarse-grid correction
(that is, contribute to the norm $> 1$) are those with medium to large singular values. As discussed previously,
it is imperative that $R$ and $P$ have a similar action on corresponding left and right singular vectors, \textit{including
large ones.} Figure \ref{fig:projs} shows that for these discretizations, it is indeed these larger singular vectors
that lead to the non-orthogonality of coarse-grid correction. 
}

\section{Discussion} \label{sec:conclusion}
\tcb{In this paper}, conditions have been established on $R$ and $P$ for two-grid and $W$-cycle multigrid convergence of 
NS-AMG in the $\sqrt{A^*A}$-norm.
Results indicate that it is not enough for $R$ and $P$ to include low-energy left and right singular vectors in their range (classical
approximation-property-based AMG approach). For a stable coarse-grid correction, the action of $R$ and $P$ must also lead to a
non-singular (and reasonably conditioned) coarse-grid operator. Sufficient conditions for this are that $R$ and $P$ accurately interpolate
singular vectors associated with small singular values, and, additionally, $R$ and $P$ have a similar action
on \textit{all} left and right singular vectors, including those associated with large singular values. 
\tcbb{An interesting open question
is the development of practical criteria that guarantee this condition.}

Furthermore, multilevel convergence of NS-AMG may require
additional iterations of relaxation or multigrid cycles on coarser levels of the hierarchy to converge,
\tcbb{depending on the strength of the approximation properties of $R$ and $P$. 
However, Theorem \ref{th:W-cycle}
indicates that, with the appropriate AMG cycle, scalable $W$-cycle convergence with respect to the number of levels in the hierarchy and problem
size is possible if the coarsening ratio is less than $1/2$.}

\tcbb{
Taking a closer look at the conditions leading up to two-level and multilevel convergence, choosing $R$ and $P$ to have stronger 
approximation properties, that is smaller constants $K_{R}$ and $K_{P}$ and larger powers $\gamma$ and $\beta$, and choosing $R$
closer to $Q^*P$, reduces the size of the stability constant $C_\Pi$ in Theorem \ref{th:stability}. This is displayed explicitly in
\eqref{eqn:Cbound} and the following discussion. Similarly, the ratio of the constants, $c_1/c_0$, relating the inner products in Section \ref{sec:multigrid:equiv} 
becomes closer to $1.0$. This, in turn, reduces the number of relaxation iterations required by \eqref{eq:nu-ML} to guarantee
$W$-cycle convergence. In the limit as $R=Q^*P$ with $\beta = 1$, the requirement becomes $\nu \geq 4K_{P}^2$. With
$\beta = 3/2$, the requirement is $\nu \geq 2K_{P}$. Appealing to Corollary \ref{cor:2grid}, in this context $K_{P} = K_{P,1,1}$.
In Section \ref{sec:numerics}, Figure 1 demonstrates that for two commonly used discretizations 
and several choices for $R$ and $P$, the SAP constants are not exceedingly large. However, the sufficient conditions derived here still require
a large number of relaxation steps. This is, in part, due to the choice of Richardson's method on $A^*A$ for relaxation. 
This choice facilitates the analysis, but forces stricter constraints than necessary 
and is probably not the best choice in practice. 
Using a similar W-cycle proof for SPD systems and Richardson's method on $A$ yields the constraint $\nu \geq 2K_P$ for $\beta = 1$.
An open question is an analysis that involves a more practical relaxation and yields less demanding sufficient conditions.
}

\tcb{To illustrate that the conditions developed here may not be necessary for NS-AMG convergence, consider the recently developed 
reduction-based method described in \cite{air1}, where sufficient conditions for $\ell^2$-convergence of the error and residual
are derived. There, conditions
for convergence are different in that a SSAP with respect to $QA$ is not necessarily required on both $R$ and $P$. Rather,
in \cite{air1}, a SSAP with respect to $QA$ (or, equivalently, a WAP with respect to $A^*A$) is required on at least one of $R$ or $P$.
The other operator then must satisfy an additional assumption on approximating the ideal restriction or ideal interpolation operator with some
level of accuracy. That being said, results in Section \ref{sec:numerics} demonstrate that the $\ell$AIR restriction operator, used to approximate ideal restriction in
\cite{air1,air2}, is also quite effective at satisfying approximation properties. Thus, it is possible these two convergence frameworks are
more related than it first appears.}

Several takeaways of the two analyses are consistent. For a robust NS-AMG solver, it is best to consider $R \neq P$. Both
theories indicate that classical AMG approaches to interpolation -- building the range of $P$ to contain error associated with small eigenvalues
-- are applicable in the nonsymmetric setting, when coupled with an appropriate restriction operator. However, care must be taken to build
$R$ and $P$ in a ``compatible'' sense, leading to a stable correction. \tcb{Numerical results in Section \ref{sec:numerics} demonstrate on a
highly nonsymmetric model problem that it is, in fact, singular vectors with larger singular values that increase the norm of the non-orthogonal
coarse-grid correction, modes which are not typically considered when forming multigrid transfer operators.}
Due to the non-orthogonal nature of NS-AMG, both analyses also indicate that modified cycles with additional
relaxation or cycling on coarser grids may be necessary for scalable convergence. The reduction-based NS-AMG algorithms developed
in \cite{air1,air2} have shown promising results on highly nonsymmetric matrices resulting from the discretization of hyperbolic PDEs.
Development of a robust  NS-AMG solver based on theory developed here is ongoing work.

\section*{Acknowledgment} 
The authors acknowledge Alyson Fox for her initial work on convergence of nonsymmetric AMG, which helped motivate some of these results. 
 
\bibliographystyle{plain}
\bibliography{NSTheory_BIB}

\newpage
\appendix
\section*{Appendix}

{\color{black}
\begin{proof}[Proof of Theorem \ref{th:FAP}]
The first part is found by noting that, for any $\kappa \geq \eta$ and $0\leq \alpha \leq \beta$
\begin{equation}
\|\mathbf{v} - P\mathbf{v}_c\|_{\mathcal{A}^{\kappa}}^2 \leq \| \mathcal{A}^{\kappa-\eta}\| \|\mathbf{v} - P\mathbf{v}_c\|_{\mathcal{A}^\eta}^2
~~\mbox{and}~~
\langle \mathcal{A}^{2\beta} \mathbf{v}, \mathbf{v}\rangle \| \leq \| \mathcal{A}^{2(\beta-\alpha)}\| \langle \mathcal{A}^{2\alpha}\mathbf{v}, \mathbf{v}\rangle.
\end{equation}
For the second result, note that if $\eta \leq \beta$ then, from the first part, P satisfies a FAP($\beta,\beta$) with constant 
$K_{P,\beta,\beta} \leq K_{P,\beta,\eta}$. Next, we prove that if $P$ satisfies a FAP($\beta,\beta$) with constant
$K_{P,\beta,\beta}$, then $P$ satisfies a FAP($\beta,0$)  with  constant $K_{P,\beta,0} \leq K_{P,\beta,\beta}^2$.

Let $\Pi_\beta$ denote the $\mathcal{A}^\beta$-orthogonal projection onto the range of $P$. By assumption
\begin{equation*}
\| (I-\Pi_\beta) \mathbf{v}\|_{A^\beta}^2 \leq \frac{ K_{P,\beta,\beta} }{\|A\|^\beta} \langle A^{2\beta} \mathbf{v},\mathbf{v}\rangle.
\end{equation*}
Let $P\mathbf{v}_c = \Pi_\beta \mathbf{v}$. 
Write,
\begin{equation}\label{eqn:lem2eqn1}
\| \mathbf{v}- P\mathbf{v}_c \|^2 = \|(I-\Pi_\beta)\mathbf{v}\|^2 =
\langle A^\beta(\mathbf{v}-P\mathbf{v}_c), A^{-\beta}(\mathbf{v}-P\mathbf{v}_c)\rangle.
\end{equation}
Now, denote $\mathbf{w} = A^{-\beta}(\mathbf{v}-P\mathbf{v}_c)$ and $P\mathbf{w}_c = \Pi_\beta \mathbf{w}$. Note that
\begin{equation*}
\langle A^\beta (\mathbf{v}-P\mathbf{v}_c), P\mathbf{z} \rangle= \langle A^\beta(I-\Pi_\beta)\mathbf{v}, P\mathbf{z}\rangle = 0,
\end{equation*}
for all $\mathbf{z}$. Applying an orthogonality argument, the Cauchy-Schwarz inequality, and a FAP($\mathcal{A},\beta,\beta)$
in the following steps, respectively, yields
\begin{align} \nonumber
\langle A^\beta(\mathbf{v}-P\mathbf{v}_c), A^{-\beta}(\mathbf{v}-P\mathbf{v}_c)\rangle&= \langle A^\beta(\mathbf{v}-P\mathbf{v}_c), A^{-\beta}(\mathbf{v}-P\mathbf{v}_c)-P\mathbf{w}_c\rangle, \\
\nonumber
&\leq \| \mathbf{v}-P\mathbf{v}_c\|_{A^\beta} \| A^{-\beta}(\mathbf{v}-P\mathbf{v}_c) -P\mathbf{w}_c \|_{A^\beta}, \\
\nonumber
&\leq \| \mathbf{v}-P\mathbf{v}_c\|_{A^\beta}\frac{\sqrt{K_{P,\beta,\beta}}}{\|A\|^{\beta/2}} \| A^\beta (A^{-\beta}(\mathbf{v}-P\mathbf{v}_c)) \|, \\
\label{eqn:lem2eqn2}
&= \| \mathbf{v}-P\mathbf{v}_c\|_{A^\beta}\frac{\sqrt{K_{P,\beta,\beta}}}{\|A\|^{\beta/2}} \| \mathbf{v}-P\mathbf{v}_c\|.
\end{align}
Combining \eqref{eqn:lem2eqn1} and \eqref{eqn:lem2eqn2} and again applying the FAP-$(A,\beta,\beta)$ yields
\begin{equation}
\|(I-\Pi_0)\mathbf{v}\|\leq\| \mathbf{v}-P\mathbf{v}_c \| \leq \frac{\sqrt{K_{P,\beta,\beta}}}{\|A\|^{\beta/2}} \| \mathbf{v}-P\mathbf{v}_c\|_{A^\beta} \leq \frac{K_{P,\beta,\beta}}{\|A\|^{\beta}} \| A^\beta \mathbf{v} \|.
\end{equation}
Thus, $P$ satisfies a FAP($\beta,0$) with constant $K_{P,\beta,0} \leq K_{P,\beta,\beta}^2 \leq K_{P\beta,\eta}^2$.

Again applying the first part, for any $0 \leq \kappa \leq \eta$,  $P$ satisfies a FAP($\beta, \kappa$) with constant
$K_{P,\beta,\kappa} \leq K_{P,\beta, 0} \leq K_{P,\beta,\beta}^2 \leq K_{P, \beta,\eta}^2$. This completes the proof.
\end{proof}
}

\begin{proof}[Proof of Lemma \ref{lem:bound_block}]
Starting with the lower bound, assume positive constants: $a_0,b,c,d_0 > 0$. An $\epsilon$-inequality can be
used to bound below in norm:
{\small
\begin{align*}
\left\|\begin{pmatrix}~~A & -B \\ -C & ~~D\end{pmatrix}\begin{pmatrix}\mathbf{x}\\\mathbf{y}\end{pmatrix}\right\|^2 & = \|A\mathbf{x} - B\mathbf{y}\|^2 + \|C\mathbf{x} - D\mathbf{y}\|^2 \\
& = \|A\mathbf{x}\|^2 - 2\langle A\mathbf{x},B\mathbf{y}\rangle + \|B\mathbf{y}\|^2 + \|C\mathbf{x}\|^2 - 2\langle C\mathbf{x},D\mathbf{y}\rangle + \|D\mathbf{y}\|^2 \\
& \geq (1 - \epsilon_1)\|A\mathbf{x}\|^2 - (\sfrac{1}{\epsilon_1}-1)\|B\mathbf{y}\|^2 + (1 - \epsilon_2)\|D\mathbf{y}\|^2 - (\sfrac{1}{\epsilon_2}-1)\|C\mathbf{x}\|^2  \\
& \geq \Big[ a_0^2(1-\epsilon_1) - c^2(\sfrac{1}{\epsilon_2}-1)\Big]\|\mathbf{x}\|^2 + 	\Big[ d_0^2(1-\epsilon_2) - b^2(\sfrac{1}{\epsilon_1}-1)\Big]\|\mathbf{y}\|^2 
\end{align*}
}
for any $\epsilon_1,\epsilon_2\in(0,1]$. Note that the upper bound on $\epsilon_1$ and $\epsilon_2$ is necessary to keep the leading
constants on $\|A\mathbf{x}\|^2$ and $\|D\mathbf{y}\|^2$ positive because we bounded these from below, and vice versa for $\|B\mathbf{y}\|^2$ and
$\|C\mathbf{x}\|^2$. This leads to a system of constraints 
\begin{align}
\begin{split}\label{eq:stab_constraint}
C_1(\epsilon_1,\epsilon_2) := a_0^2(1-\epsilon_1) - c^2(\sfrac{1}{\epsilon_2}-1) & > 0, \\
C_2(\epsilon_1,\epsilon_2) := d_0^2(1-\epsilon_2) - b^2(\sfrac{1}{\epsilon_1}-1) & > 0,
\end{split}
\end{align}
for some $\epsilon_1,\epsilon_2\in(0,1]$. The boundary of these constraints in the $(\epsilon_1,\epsilon_2)$-plane is
given by the functions 
\begin{align*}
\widehat{\epsilon_2}(\epsilon_1)  = \frac{c^2}{c^2 + a_0^2(1-\epsilon_1)}, \hspace{4ex}
\widetilde{\epsilon_2}(\epsilon_1) = 1 + \frac{b^2}{d_0^2} - \frac{b^2}{d_0^2\epsilon_1},
\end{align*}
with the region of points satisfying the constraints bounded below by $\widehat{\epsilon_2}$ and above by $\widetilde{\epsilon_2}$. 
A little algebra shows that $\widehat{\epsilon_2}$ is concave up, $\widetilde{\epsilon_2}$ concave down, and both functions are monotonically
increasing over $(0,1]$ with a crossover point at $\widehat{\epsilon_2}(1) = \widetilde{\epsilon_2}(1) = 1$. It follows that there exists
some region within $(0,1)\times(0,1)$ (constraints on $\epsilon_1$ and $\epsilon_2$) that satisfies \eqref{eq:stab_constraint} if and only if
$\widehat{\epsilon_2}'(1) > \widetilde{\epsilon_2}'(1)$, which reduces to $a_0d_0 > bc.$

The maximum bound is obtained by setting the leading constants on $\|\mathbf{x}\|^2$ and $\|\mathbf{y}\|^2$ equal. Thus we will consider a constrained
maximization over $C_1$ such that $C_1 = C_2$ (or vice versa). Since we are maximizing the intersection of two convex functionals,
which is also convex, the maximum is unique. Thus consider $\epsilon_2(\epsilon_1)$ and denote $\epsilon_2' :=
\frac{\partial\epsilon_2}{\partial\epsilon_1}$. Then, at the maximum, we must have $\frac{\partial}{\partial\epsilon_1}
C_1(\epsilon_1,\epsilon_2(\epsilon_1)) = \frac{\partial}{\partial\epsilon_1} C_1(\epsilon_1,\epsilon_2(\epsilon_1)) = 0$:
\begin{align*}
-a_0^2 + \frac{c^2}{\epsilon_2^2}\epsilon_2' & = 0 \hspace{3ex}\implies\hspace{3ex} \epsilon_2' = \frac{a_0^2}{c^2}\epsilon_2^2, \\
-d_0^2 \epsilon_2' + \frac{b^2}{\epsilon_1^2} & = 0 \hspace{3ex}\implies\hspace{3ex} \epsilon_2' = \frac{b^2}{d_0^2\epsilon_1^2}.
\end{align*}
Setting the functions for $\epsilon_2'$ equal leads to the constraint $\epsilon_2 = \frac{bc}{a_0d_0\epsilon_1}$, and plugging into
$C_1$ and $C_2$ gives
\begin{align*}
C_1(\epsilon_1) & = a_0^2 + c^2 - \epsilon_1\left(a_0^2 + \frac{a_0cd_0}{b}\right), \\
C_2(\epsilon_1) & = d_0^2 + b^2 - \frac{1}{\epsilon_1}\left(\frac{bcd_0}{a_0} + b^2\right).
\end{align*}
Setting $C_1 = C_2$ leads to a quadratic function in $\epsilon_1$:
\begin{align*}
\epsilon_1^2\left(a_0^2 + \frac{a_0cd_0}{b}\right) + \epsilon_1\left(b^2 + d_0^2 - a_0^2 - c^2\right) - b\left(\frac{cd_0}{a_0} + b\right) & = 0.
\end{align*}
Because $a_0,b,c,d_0 > 0$, we have $-b\left(\frac{cd_0}{a_0} + b\right) < 0$ and, thus, there exists exactly one positive root, given by
\begin{align*}
\epsilon_1 & = \frac{(a_0^2+c^2-b^2-d_0^2) + \sqrt{(a_0^2+c^2-b^2-d_0^2)^2 + 4(a_0b+cd_0)^2}}{2\left(a^2 + \frac{a_0cd_0}{b}\right)}.
\end{align*}
Plugging into $C_1$ gives
\begin{align}
C_1(\epsilon_1) & = C_2(\epsilon_1) = \frac{a_0^2 + b^2 + c^2 + d_0^2 - \sqrt{(a_0^2+c^2-b^2-d_0^2)^2 + 4(a_0b+cd_0)^2}}{2},\label{eq:lower_bound}
\end{align}
where $\eta_0 := C_1(\epsilon_1)$. Setting $b=0$ or $c=0$ and repeating the above process leads to a lower bound consistent with
setting $b=0$ or $c=0$ in \eqref{eq:lower_bound}.

A similar derivation can be used for an upper bound. Let us start by assuming positive bounds, $a_1,b,c,d_1>0$. We bound in norm
from above, again using an $\epsilon$-inequality, and seek to minimize the intersection of 
\begin{align*}
C_3(\epsilon_1,\epsilon_2) := a_1^2(1+\epsilon_1) + c^2(1 + \sfrac{1}{\epsilon_2}),\\
C_4(\epsilon_1,\epsilon_2) := d_1^2(1+\epsilon_2) + b^2(1 + \sfrac{1}{\epsilon_1}).
\end{align*}
Each of these are concave up, convex functionals in the positive $(\epsilon_1,\epsilon_2)$-plane (note, there are no
constraints on the constants for this region to exist), and a minimum is attained when$\frac{\partial}{\partial\epsilon_1}
C_3(\epsilon_1,\epsilon_2(\epsilon_1)) = \frac{\partial}{\partial\epsilon_1} C_4(\epsilon_1,\epsilon_2(\epsilon_1)) = 0$.
This leads to a quadratic functional in $\epsilon_1$:
\begin{align*}
\epsilon_1^2\left(a_1^2 + \frac{a_1cd_1}{b}\right) + \epsilon_1\left(a_1^2 + c^2 - b^2 - d_1^2\right) - b\left(\frac{cd_1}{a_1} + b\right) & = 0,
\end{align*}
with one positive root by Descartes' rule of signs and the assumption $a_1,b,c,d_1>0$. The root is given by 
\begin{align*}
\epsilon_1 & = \frac{(b^2+d_1^2-a_1^2-c_1^2) + \sqrt{(a_1^2+c^2-b^2-d_1^2)^2 + 4(a_1b+cd_1)^2}}{2\left(a_1^2 + \frac{a_1cd_1}{b}\right)},
\end{align*}
which we can plug into $C_3$ and $C_4$ to solve for an upper bound
\begin{align}
\eta_1 = \frac{a_1^2 + b^2 + c^2 + d_1^2 + \sqrt{(a_1^2+c^2-b^2-d_1^2)^2 + 4(a_1b+cd_1)^2}}{2}.\label{eq:upper_bound}
\end{align}

In the case that some of $a_1,b,c$, or $d_1$ are equal to zero, it is straightforward to use a single $\epsilon$-inequality
to derive an upper bound, and verify that this bound is equivalent to plugging the appropriate zeros into \eqref{eq:upper_bound}.
\end{proof}

\end{document}